\documentclass[11pt,leqno]{article}
\usepackage{amssymb,amsbsy,amsmath,amsfonts,amssymb,amscd, mathrsfs}
\usepackage[english]{babel}
\usepackage[T1]{fontenc}
\usepackage{indentfirst}

\usepackage{amsmath}
\usepackage{amsthm}
\usepackage[pdfborder={0 0 0}]{hyperref}

\usepackage{amsfonts}
\usepackage{amssymb}
\usepackage{mathrsfs}

\numberwithin{equation}{section}
\DeclareMathOperator{\im}{Im}
\newcommand{\Cl}{\mathbb{C}}
\newcommand{\Rl}{\mathbb{R}}
\newcommand{\Nl}{\mathbb{N}}

\newtheorem{theorem}{Theorem}[section]
\newtheorem{proposition}[theorem]{Proposition}
\newtheorem{lemma}[theorem]{Lemma}

\newtheorem{definition}[theorem]{Definition}

\usepackage{enumerate}

\newcommand{\Om}{\Omega}

\newcommand{\te}{\theta}
\newcommand{\la}{\lambda}
\newcommand{\La}{\Lambda}
\newcommand{\PG}{\mathcal{P}}
\newcommand{\Disc}{\mathbb{D}}
\newcommand{\Hol}{{\rm Hol}}
\newcommand{\arctanh}{\text{arctanh}}
\newcommand{\SO}{\text{SO}}
\newcommand{\Sp}{\text{Sp}\,}
\newcommand{\dom}{\mathrm{dom}\,}
\newcommand{\re}{\mathrm{Re}}
\newcommand{\Strip}{\mathbb{S}}
\newcommand{\Dc}{\mathcal{D}}
\newcommand{\B}{\mathcal{B}}

\newcommand{\NN}{\mathcal{N}}
\newcommand{\Hil}{\mathcal{H}}
\newcommand{\DD}{\mathcal{D}}

\newcommand{\bno}[1]{|\!|\!|#1|\!|\!|}

\newcommand\C{\mathbb C}
\newcommand\R{\mathbb R}
\newcommand\T{\mathbb T}
\newcommand\D{\mathbb D}

\newcommand\eps{\varepsilon}
\newcommand\ind{{\rm 1\kern-.30em I}}

\let\phi=\varphi

\title{Approximation numbers of\\ weighted  composition operators}
\author{G.~Lechner\thanks{Cardiff University, School of Mathematics, LechnerG@cardiff.ac.uk},
\;\,
D.~Li\thanks{Universit\'e d'Artois, Laboratoire de Math\'ematiques de Lens (LML), daniel.li@euler.univ-artois.fr}, 
\;\,
H.~Queff\'elec\thanks{Universit\'e Lille Nord de France, Herve.Queffelec@univ-lille1.fr},
\;\,
L.~Rodr\'iguez-Piazza\thanks{Universidad de Sevilla, piazza@us.es}}

\date{December 24, 2017}

\begin{document}

\maketitle

\begin{abstract}
     We study the approximation numbers of weighted composition operators $f\mapsto w\cdot(f\circ\varphi)$ on the Hardy space $H^2$ on the unit disc. For general classes of such operators, upper and lower bounds on their approximation numbers are derived. For the special class of weighted lens map composition operators with specific weights, we show how much the weight $w$ can improve the decay rate of the approximation numbers, and give sharp upper and lower bounds. These examples are motivated from applications to the analysis of relative commutants of special inclusions of von Neumann algebras appearing in quantum field theory (Borchers triples).
\end{abstract}

\vspace{5mm}

\section{Introduction}

In the study of composition operators $C_\varphi:f\mapsto f\circ\varphi$ acting on a Hilbert space $H$ of analytic functions (on the unit disk $\D$), one is typically interested in understanding how function-theoretic properties of $\varphi$ are related to operator-theoretic properties of $C_\varphi$. Basic properties such as boundedness or compactness of $C_\varphi$ are by now well characterized in terms of $\varphi$ in many cases \cite{Shapiro:1993,CowenMacCluer:1994}.  More recently, also the membership of $C_\varphi$ in various smaller ideals~$I$ of bounded operators on $H$ (such as the $p$-Schatten class), and more precisely the behavior of the approximation numbers $a_n(C_\varphi)$ of $C_\varphi$, was studied in depth in several papers (see e.g.  \cite{LiQueffelecRodriguez-Piazza:2012,LefevreLiQueffelecRodriguez-Piazza:2012,LiQueffelecRodriguez-Piazza:2012_3,LQR}). 

If  $\mathcal{M}(H)$ denotes the space of  multipliers of $H$ (those $w\in H$ such that $wf\in H$ for each $f\in H$), we can twist  a composition operator $C_\varphi$, assumed to map $H$ to itself,  by composing it on the left with the operator $M_w$ of multiplication by  $w\in \mathcal{M}(H)$. We then get a so-called {\em weighted composition operator} $T=M_w\,C_\varphi$ (see e.g. \cite{LEF} or \cite{CGP,HLNS}).

A careful distinction must be made between the multipliers of $H$, denoted $\mathcal{M}(H)$, and those of $C_{\varphi}(H)$,  denoted  $\mathcal{M}(H,\varphi)$, namely those functions  $w\in H$ such that $wf\in H$ for each $f$ belonging to  the range $C_{\varphi}(H)$, not necessarily to the whole of $H$. For example, if $H=H^2$ is the Hardy space, then $\mathcal{M}(H)=H^\infty$ consists of all bounded analytic functions on $\D$. It can be proved that (see \cite{AT,COHE} and \cite{GKP}, respectively):
\begin{align*}
	\mathcal{M}(H^2,\varphi)
	&=
	H^\infty\Leftrightarrow \varphi \hbox{\ is a finite Blaschke product},
	\\
	\mathcal{M}(H^2,\varphi)
	&=
	H^2\,\,\Leftrightarrow \Vert\varphi\Vert_\infty<1.
\end{align*}

In this paper, we study approximation numbers of weighted composition operators in the case of the Hardy space $H=H^2$ of the disc, and a weight $w  \in\mathcal{M}(H)\subset \mathcal{M}(H,\varphi)$. Then, since we are dealing with ideals, twisting with the bounded operator $M_w$ can but reinforce the membership in $I$, and improve the rate of decay of approximation numbers. 

\bigskip

There are (at least) three motivations for considering weighted composition operators: First, they form a natural and non-trivial generalization of composition operators on $H^2(\Disc)$. In this context, it is natural to ask how much faster the approximation numbers $a_n(M_wC_\varphi)$ can decay in comparison to the $a_n(C_\varphi)$. For example, can $M_wC_\varphi$ be compact when $C_\varphi$ is non-compact, or can the $a_n(M_wC_\varphi)$ decay quite fast when the $a_n(C_\varphi)$ decay rather slowly? We will address these questions in the body of the text\footnote{Another application of weighted composition operators to the study of composition operators on spaces of several complex variables can be found in \cite{LIQR}.}.

As a second motivation, suppose $C^G_\varphi$ is a composition operator on a Hardy space $H^2(G)$ over a simply connected region properly contained in $\Cl$. Then a choice of Riemann map $\tau:\Disc\to G$ induces a unitary between $H^2(\Disc)$ and $H^2(G)$ \cite{Duren:1970}, and we can equivalently formulate $C^G_\varphi$ as an operator on $H^2(\Disc)$. This operator on $H^2(\Disc)$, however, turns out to be a {\em weighted} composition operator $M_{w_\tau}C^\Disc_{\varphi_\tau}$ in general (see \cite{ShapiroSmith:2003} and Section~\ref{Section:modular}). Thus composition operators on domains other than $\Disc$ automatically produce weighted composition operators on $H^2(\Disc)$.

A third motivation for studying weighted composition operators comes from applications in a completely different field, namely inclusions of von Neumann algebras, used in mathematical physics to model quantum field theories \cite{Haag:1996}. For $\NN$ a von Neumann algebra with a cyclic and separating vector $\Om$ on a Hilbert space $\Hil$, we will consider the Hilbert space $\DD$ obtained by closing the domain of the modular operator of $(\NN,\Om)$ \cite{Takesaki:2003} in its graph norm. If $\NN$ carries additional structure (a Borchers triple), this setting is related to complex analysis because an irreducible component of $\DD$ can be naturally identified with a Hardy space $H^2(\Strip)$ on a strip region $\Strip\subset\Cl$, bounded by two lines parallel to $\Rl$ (see Section~\ref{Section:modular}).

In applications in mathematical physics, one is interested in specific inclusions $\tilde\NN\subset\NN$ and their relative commutants, the size of which can be controlled if a map built from the modular operator has sufficiently quickly decaying approximation numbers \cite{BuchholzDAntoniLongo:1990-1}. On the level of the irreducible component giving rise to the Hardy space $H^2(\Strip)$, this condition translates to a weighted restriction operator $R_w:H^2(\Strip)\to L^2(\Rl)$, $f\mapsto (w\cdot f)|_\Rl$, where the real line~$\Rl$ lies in the interior of the strip, and $w\in H^\infty(\Strip)$ is an inner function on $\Strip$ obtained from the inclusion $\tilde\NN\subset\NN$. For the application in physics, sharp upper bounds on the approximation numbers of $R_w$ are desirable \cite{AlazzawiLechner:2016}.

\smallskip

Mapping the strip $\Strip$ to the disc, $R_w$ can be formulated as a Carleson embedding operator (the definition of which we recall in Section~\ref{section:preliminaries}). These operators are often used in estimating approximation numbers of composition operators \cite{LefevreLiQueffelecRodriguez-Piazza:2012,LiQueffelecRodriguez-Piazza:2012_3}. In turn, we find from the strip picture that the embedding operator can be estimated from above by special weighted composition operators on the disc, namely those whose symbol is a lens map $\varphi=\varphi_{\la}$, $0<\la<1$ (see Section~\ref{Section:lensmaps} for the definition). This closes the connection to composition operators on $H^2(\Disc)$, where the $C_{\varphi_\la}$ are among the best studied examples \cite{LefevreLiQueffelecRodriguez-Piazza:2012}.

\bigskip

Given these motivations, this article is organized as follows. In Section~\ref{Section:general}, we introduce our notation and setup, and study weighted composition operators with general symbols $\varphi$ and weights $w$ on the disc. After deriving a simple upper bound, we give an example how a weight can turn a non-compact composition operator into a compact one. Regarding lower bounds, we show that the worst possible behavior of the $a_n(C_\varphi)$ (exponential if $\Vert \varphi\Vert_\infty<1$, subexponential if $\Vert \varphi\Vert_\infty=1$) is the same for the weighted operators $M_wC_\varphi$.

\medskip

In Section~\ref{Section:modular}, we explain the links between modular theory of von Neumann algebras, Hardy spaces on strips, and weighted restriction operators. In that section we also show how weighted lens map composition operators appear. Section~\ref{Section:modular} can be read independently of the other parts of the article.

\medskip

Finally, we consider in Section~\ref{Section:lensmaps} the specific case of weighted lens map composition operators $M_wC_{\varphi_\la}$ as our primary example. In the case without weight, the approximation numbers of $C_{\varphi_\la}$ are known to decay like $e^{-c\sqrt n}$~\cite{LiQueffelecRodriguez-Piazza:2012_3}. A natural question in this context is how close the decay rate of the $a_n(M_wC_{\varphi_\la})$ can come to exponential decay $e^{-cn}$ (the optimal one in the context of composition operators). For the weights motivated by the considerations in Section~\ref{Section:modular}, we show that $a_n(M_wC_{\varphi_\la})$ decays like $e^{-c\frac{n}{\log n}}$.

\section{Weighted composition operators on $\Disc$}\label{Section:general}

\subsection{Preliminaries}\label{section:preliminaries}

We begin by recalling a few operator-theoretic and function-theoretic  facts. The approximation numbers  $a_{n}(T)=a_n$ of an operator $T:H\to H$ (with $H$ a Hilbert space) are defined by 
$$a_n=\inf_{{\rm rank} R<n} \Vert T-R\Vert\,,$$
and $T$ is compact if and only if $\lim_{n\to \infty}a_{n}(T)=0$. 
According to a result of Allahverdiev \cite[p.~155]{CAST}, $a_n=s_n$, the $n$-th singular number of $T$. 
We have the following alternative definition (a variant of Kolmogorov numbers) of $a_{n}(T)$ \cite{LQR}, in which $B_H$ denotes the closed unit ball of $H$ and $d(g,A)$ the distance of $g$ to  $A\subset H$: 
\begin{equation}\label{alter} a_{n}(T)=\inf_{\dim E<n}\Big[\sup_{f\in B_H} d(Tf,\ TE)\Big].\end{equation}
The definition of $a_{n}(T)$ also makes sense for $T:X\to Y$ an operator between Banach spaces (see Theorem\ \ref{extension}  to come).\\
Coming back to the hilbertian setting, two other useful alternative definitions (respectively in terms of \textit{Bernstein} and \textit{Gelfand} numbers) are, denoting by $S_E$ the unit sphere of a subspace $E$ of $H$ (see \cite[Chapter 2]{CAST}, or \cite{LiQueffelecRodriguez-Piazza:2012}):
\begin{align}
	\label{serge} 
	a_{n}(T)
	&=
	\sup_{\dim E=n}\Big[\inf_{f\in S_E} \Vert Tf\Vert\Big]\,,
	\\
	\label{gelf} 
	a_{n}(T)
	&=
	\inf_{{\rm codim} E <n} \Vert T_{|E}\Vert=\inf_{{\rm codim} E <n}\Big[\sup_{f\in S_E} \Vert Tf\Vert\Big]\,.
\end{align}

The following parameters $0\leq \beta^{-}(T) \leq \beta^{+}(T)\leq 1$ were used in \cite{LQR}: 
$$\beta^{+}(T)=\limsup_{n\to \infty}\big[a_{n}(T)\big]^{1/n},\qquad \beta^{-}(T)=\liminf_{n\to \infty}\big[a_{n}(T)\big]^{1/n}.$$
When the limit exists, we denote it by $\beta(T)$. It is proved in \cite{LQR} that this is the case for $T$ a composition operator on the Hardy, Bergman, or Dirichlet space. Observe that $\beta^{-}(T)=1$ signifies a subexponential decay for $a_{n}(T)$, namely  $a_{n}(T)\geq e^{-n\eps_n}$ where $\eps_n>0$ and $\eps_n\to 0$.

Let  now $\D$ be the open unit disk of the complex plane and $H^2$ the usual Hardy space of $\D$. Recall \cite[p.~12]{Shapiro:1993} that the norm of $f(z)=\sum_{n=0}^\infty f_n z^n\in H^2$ is defined by $\Vert f\Vert_{2}^{2}=\sum_{n=0}^\infty |f_n|^2$, or alternatively by 
\begin{equation}\label{alt} \Vert f\Vert_{2}^{2}=\int_{\T} |f^{\ast}(u)|^2 dm(u)\end{equation}
where  $m$ denotes the Haar measure of the unit circle $\T$ and  $f^{\ast}(u)$ is the ($m$-almost everywhere existing by Fatou's theorem) radial limit $\lim_{r\to 1^{-}} f(ru)$, often again denoted $f(u)$. \\
The space of multipliers of $H^2$ is isometrically isomorphic to the space $H^\infty$ of functions analytic and bounded on $\D$ \cite{Garnett:2007}. This means that any function $w\in H^\infty$ defines a bounded multiplication operator $M_w:H^2\to H^2$ by the formula $M_{w}(f)=wf$ and that 
$$\Vert M_w\Vert:=\sup_{\Vert f\Vert_2\leq 1} \Vert wf\Vert_2=\Vert w\Vert_\infty:=\sup_{z\in \D} |w(z)|.$$  If $\varphi$ is a \textit{non-constant} and analytic self-map of $\D$ (often called a symbol), the associated composition operator $C_\varphi:H^2\to H^2$ is defined by 
$$C_{\varphi}(f)=f\circ \varphi.$$
The fact that $C_\varphi$ boundedly maps $H^2$ to itself for any symbol $\varphi$ is the well-known subordination principle of Littlewood (\cite[p.~29]{CowenMacCluer:1994}, \cite[p.~16]{Shapiro:1993}).\\

Next, a positive and bounded measure $\mu$ on $\D$ is called a Carleson measure (for $H^2$) if the identity map $R_\mu, R_{\mu}(f)=f$ maps  $H=H^2$ to $L^{2}(\mu)$, that is if there exists a constant $C$ such that: 
$$\int_{\D} |f(z)|^2 d\mu(z)\leq C\Vert f\Vert_{2}^2\quad  \forall f\in H^2.$$
The best constant $C$ is called the Carleson-norm of $\mu$ and is denoted $\Vert \mu\Vert_{\mathcal{C}}$. That is $\Vert \mu\Vert_{\mathcal{C}}=\Vert R_\mu\Vert^2$. Let us set 
$$\rho_{\mu}(h):=\sup_{\xi\in \T} \mu[S(\xi,h)]$$ where, for $\xi\in \T=\partial\Disc$, $S(\xi, h)$ is the Carleson box 
$$S(\xi,h)=\{z\in \D : |z-\xi|\leq h\}.$$
With those notations, the Carleson embedding theorem \cite[p.~37]{CowenMacCluer:1994} gives a geometric characterization of Carleson measures:

\begin{theorem} \label{thm:carleson}
	Let $\mu$ be a positive and bounded measure  on $\D$. Then, $\mu$ is a Carleson measure if and only if, for some constant $K$, 
	$$\rho_{\mu}(h)\leq Kh\quad \forall h\in ]0,1].$$
	In this case, $\Vert \mu\Vert_{\mathcal{C}}\leq aK$ where $a>0$ is an absolute constant.
\end{theorem}

Let again $\varphi^{\ast}(u)=\lim_{r\to 1^{-}}\varphi(ru)$. Littlewood's subordination principle implies that $m_\varphi=\varphi^{\ast}(m)$, the image under $\varphi^{\ast}$ of this Haar measure, is a Carleson measure for $H^2$, and we will write $\rho_{\varphi}$ instead of $\rho_{m_{\varphi}}$. We will also write  $\Vert . \Vert$ instead of  $\Vert . \Vert_2$ when there is no ambiguity.\\ 

 The main subject of this article are {\em weighted composition operators} on $H^2$, defined in terms of a weight $w$ (typically $w\in H^\infty$) and a symbol $\varphi$, according to $T_{w,\varphi}:=M_w\, C_\varphi$. Often we will denote this operator by $T$ for short.\\
 
Depending on $\varphi$, the operator $C_\varphi$ can be compact or not. As we already said, passing from $C_\varphi$ to $M_wC_\varphi$ can only improve this compactness, or the behavior of singular numbers, thanks to the ideal property of both notions. It is the purpose of this paper to investigate the question more closely.

\subsection{A simple general upper bound} 

Throughout this section and the rest of this paper, we use the notation $A\lesssim B$ (resp.~$A\gtrsim B$) to indicate that $A\leq \lambda B$ (resp.~$A\geq \lambda B$) where $\lambda$ is a uniform positive  constant, ``uniform'' being clear from the context.  Let $\varphi$ be a symbol continuous on $\overline{\D}$, fixing $1$, and $w_0\in H^\infty$  a weight. We set $\gamma(t)=\varphi(e^{it})$. We assume that $\varphi(\D)$ has no other contact points than 
$1$ with $\T$ and more precisely that:
\begin{equation}\label{secarte} 
     |t|\leq \pi\Rightarrow 1-|\gamma(t)|\geq \omega(|t|)\,,
\end{equation}
where $\omega:[0,\pi]\to \R^{+}$ is an increasing function with $\omega(0)=0$. We also set 
\begin{equation}\label{alse} \delta_{w_0}(h)=\sup_{|t|\leq \omega^{-1}(h)} |w_{0}(\gamma(t))|.\end{equation}
 We then have an upper bound for a special class of weights related to $\varphi$:
\begin{theorem}\label{geup} Let $\varphi$ be a symbol satisfying (\ref{secarte}), $w_0$ a weight,  $w=w_0\circ \varphi$,  $T=M_w\,C_\varphi$,  and $a_n=a_{n}(T)$. Then: 
\begin{equation}\label{irem}a_{n}\lesssim \inf_{0<h<1}\Big[e^{-nh}+\delta_{w_0}(h)\Big]=:\rho_n.\end{equation}
\end{theorem}
\begin{proof}The proof is close to that of \cite[Thm.~5.1]{LiQueffelecRodriguez-Piazza:2012_3}. Let $f=z^n g\in z^nH^2=:E$, a subspace of $H^2$ of codimension $n<n+1$. Assume that $\Vert f\Vert=1$, so that $\Vert g\Vert=1$. We see that, given $0<h<1$: 
\begin{align*}
		\Vert T(f)\Vert^2
		&=
		\int_{\T} |w_{0}(\varphi(u))|^2  |\varphi(u)|^{2n} |g(\varphi(u))|^2 dm(u)
		\\
		&=
		\int_{\D} |w_{0}(z)|^2|z|^{2n}  |g(z)|^2 dm_{\varphi}(z)
		\\
		&=
		\int_{(1-h)\overline{\D}}   |w_{0}(z)|^2|z|^{2n} |g(z)|^2 dm_{\varphi}(z)+\int_{\D}  |z|^{2n} |g(z)|^2 d\mu_{h}(z)
		\\
		&\lesssim 
		(1-h)^{2n}+\Vert \mu_h\Vert_{\mathcal{C}}\lesssim e^{-2nh}+\Vert \mu_h\Vert_{\mathcal{C}}\,,
\end{align*}
where $\mu_h$ denotes the restriction (trace) of the measure $|w_0|^2 dm_{\varphi}$ to the annulus $A_h=\{z:1-h<|z|<1\}$.
It remains to estimate $\Vert \mu_h\Vert_{\mathcal{C}}$, which we do through Carleson's embedding theorem.

Since $\mu_h$ is carried by $A_h$, we can consider only boxes $S=S(\xi,r)$ with $0<r\leq h$. Then 
\begin{align*}
	I_r
	:=&
	\int1_{S}(u)d\mu_{h}(u)=\int1_{S}(u)|w_{0}(u)|^2 dm_{\varphi}(u)
	\\
	=&
	\int_{S\cap\varphi(\partial{\D})}|w_{0}(u)|^2dm_{\varphi}(u)
	\leq 
	\Big(\sup_{S\cap\varphi(\partial{\D})}|w_{0}(u)|^2\Big)\,m_{\varphi}(S)
\end{align*}
since $m_\varphi$ is carried by $\varphi(\partial{\D})$. Now, if $u=\varphi(e^{it})\in S$ and $|t|\leq \pi$, then 
$$\omega(|t|)\leq1-|\gamma(t)|\leq |\xi-\gamma(t)|\leq r,$$ that is $|t|\leq \omega^{-1}(r)$. 
Therefore, $I_r\leq \delta_{w_0}^{2}(r)\rho_{\varphi}(r)\lesssim r \delta_{w_0}^{2}(h)$, and  by Carleson's  theorem:
$$\Vert \mu_h\Vert_{\mathcal{C}}\lesssim \sup_{0<r<h}\frac{I_r}{r}\lesssim \delta_{w_0}^{2}(h),$$
giving $a_{n+1}\lesssim \rho_n$ in view of the alternative definition (\ref{gelf}) of approximation numbers, and ending the proof after a change of $n+1$ to $n$.
\end{proof}

\subsection{From non-compactness to compactness}

It is known that twisting a non-compact composition operator $C_\varphi$ with a multiplication operator $M_w$ can result in $M_wC_\varphi$ being compact (see, for example, \cite{GKP}). We now give a first application of Theorem \ref{geup}, in which this effect is demonstrated in terms of explicit estimates on approximation numbers, which seems to be new. 

Note that the compactness result of part $ii)$ of the following theorem also follows by applying \cite[Thm.~2.8]{GKP} to the specified symbol $\varphi$ and weight~$w$.

\begin{theorem}\label{thm:disc-in-disc}
	Let $\varphi(z)=\frac{1+z}{2}$ and $w(z)=(1-z)^\alpha,\  \alpha>0$. Then: 
	\begin{enumerate}
		\item $C_\varphi$ is non-compact and indeed $\Vert C_\varphi\Vert_e=\Vert C_\varphi\Vert=\sqrt 2$. 
		\item $T=M_w\,C_\varphi$ is compact and its approximation numbers verify\\ 
		$a_{n}(T)\lesssim\big(\frac{\log n}{n}\big)^{\alpha/2}$.
	\end{enumerate}
	In particular, a weighted composition operator  $T=M_w C_\varphi$ can be compact while its ``compositional symbol'' $\varphi$ has no fixed point inside~$\D$.
\end{theorem}
\begin{proof}Recall that $\Vert C_\varphi\Vert_e=:\lim_{n\to \infty}a_{n}(C_\varphi)$ is the essential norm of $C_\varphi$. The first item i) is well-known (\cite{Shapiro:1987}, see also \cite{CLDA}). For upper bounds, we may  use Theorem \ref{geup} since   $w=w_0\circ \varphi$ where $w_{0}(z)=2^{\alpha}\,(1-z)^\alpha$.\\
 Next,  we observe that, for $|t|\leq \pi$:
$1-|\gamma(t)|=1-\cos(t/2)=2\sin^{2}(t/4)\geq  \delta t^2$, so that, up to absolute constants, we are allowed to take $\omega(h)=h^2$ and $\omega^{-1}(h)=\sqrt h$ in Theorem \ref{geup}. Since $|w_{0}(\gamma(t))| \lesssim|1-\gamma(t)|^\alpha\leq |t|^\alpha$, this implies that $\delta_{w_0}(h)\lesssim h^{\alpha/2}$, and subsequently that 
$$a_{n}(T)\lesssim   \inf_{0<h<1}\Big[e^{-nh}+\delta_{w_0}(h)\Big] \lesssim \inf_{0<h<1}\Big[e^{-nh}+h^{\alpha/2}\Big]\lesssim \Big(\frac{\log n}{n}\Big)^{\alpha/2}$$
by taking $h=C\frac{\log n}{n}$ where $C$ is a large numerical constant.  This gives the claimed upper bound. \\
 Finally, the fixed point of $\varphi$ is $1$ and $1\notin \D$.
\end{proof} 
\noindent{\bf Remark:} The simple estimates given here are  not sharp (Theorem \ref{onemain} will give sharper results) and are  just intended to show that multiplication by $w$ can improve the decay of approximation numbers. In particular, the logarithmic factor can be dropped in the example of Thm.~\ref{thm:disc-in-disc}, and whereas the estimate in Theorem~\ref{thm:disc-in-disc}~{\em ii)} give membership in the Hilbert-Schmidt class for $\alpha>1$, one actually has the following stronger statement.
\newpage
\begin{proposition}
	The following are equivalent in the previous example: \begin{enumerate}
		\item $M_{w}C_\varphi$ is Hilbert-Schmidt.
		\item $\alpha>1/2$.
	\end{enumerate}
\end{proposition}
\begin{proof}If $T=M_w C_\varphi$ and if $(e_n)$ is the canonical basis of $H^2$, we find that
$$\sum_{n=0}^\infty \Vert T(e_n)\Vert^2=2\int_{\T} \frac{|w(e^{i\theta})|^2}{1-\cos \theta}d\theta\approx \int_{0}^\pi \theta^{2\alpha-2}d\theta$$ and the latter integral is finite iff $\alpha>1/2$. Alternatively, we could use $\Vert T(e_n)\Vert\approx n^{-\alpha/2-1/4}$.
\end{proof}

\subsection{Maximal  general possible decay}

It was proved in \cite{LiQueffelecRodriguez-Piazza:2012_3} that singular numbers of composition operators never have a superexponential decay. The same holds for weighted composition operators.

\begin{theorem}\label{saho}Let $T=M_w\,C_\varphi$ be a weighted composition operator. Then $\beta^{-}(T)>0$, that is, there exist positive constants $\delta$ and $\rho$ such that for any integer $n\geq 1$: 
$$a_{n}(T)\geq \delta \rho^n.$$
\end{theorem}
\begin{proof}We first recall that the interpolation constant $I_z$ of a (finite or not) sequence $z=(z_j)$ of distinct points of $\D$ is the smallest constant $K$ such that, for any bounded  sequence $c=(c_j)$, one can find $h\in H^\infty$ such that 
$$h(z_j)=c_j \quad \forall j \hbox{\quad and}\quad \Vert h\Vert_\infty\leq K\sup_{j}|c_j|.$$
The connection between interpolation constants and reproducing kernels is given by the well-known two-sided inequality \cite[p.~302-303]{Nikolskii:2002}, valid for all scalars $\lambda_j$:

\begin{equation}\label{wk}I_{z}^{-2}\sum_j |\lambda_j|^2 \Vert K_{z_j}\Vert^2\leq \Big\Vert \sum_j \lambda_j K_{z_j}\Big\Vert^2\leq I_{z}^{2}\sum_j |\lambda_j|^2 \Vert K_{z_j}\Vert^2.\end{equation}
We shall now rely on the following lemma:
\begin{lemma}\label{berber} Let $u=(u_j)_{1\leq j\leq n}$ be a sequence of length $n$ of points of $\D$  and $v=(v_j)=(\varphi(u_j))_{1\leq j\leq n}$. We assume that  the $v_j$'s (and hence the $u_j$'s) are distinct. Let $I_v$ be the interpolation constant  of $v$. Then  if \ $T=M_w\, C_\varphi$:
$$a_{n}(T)\gtrsim \big(\inf_{1\leq j\leq n} |w(u_j)|\big)\times \big( \inf_{1\leq j\leq n} \sqrt{1-|u_j|^2}\big)\times I_{v}^{-2}.$$
\end{lemma}

Indeed, we use the ``model space'' $E$ generated by the reproducing kernels $K_{u_1},\ldots, K_{u_n}$ of $H^2$ as well as the mapping equation for weighted composition operators
\begin{equation}\label{mapequ} T^{\ast}(K_a)=\overline{w(a)} K_{\varphi(a)},\end{equation}
a well-known and readily verified fact since for all $g\in H^2$:
\begin{align*}
     \langle g,T^{\ast}(K_a)\rangle
     &=
     \langle Tg, K_a\rangle= \langle w\,(g\circ \varphi), K_a\rangle
     \\
     &=w(a)\,g(\varphi(a))=w(a) \langle g, K_{\varphi(a)}\rangle
     =\langle g,\overline{w(a)}K_{\varphi(a)}\rangle.
\end{align*}

Now,  (\ref{mapequ}) and estimates analog to those of \cite{LiQueffelecRodriguez-Piazza:2012} prove the lemma. We provide some details. Let $f=\sum_{j=1}^n \lambda_j K_{u_j}\in S_E$, the unit sphere of $E$. So that, using (\ref{wk}) for the finite sequence $u=(u_j)_{1\leq j\leq n}$\ : 
\begin{equation}\label{unus}1\leq I_{u}^{2}\, \sum_{j=1}^n |\lambda_j|^2 \Vert K_{u_j}\Vert^2.\end{equation}
Similarly, since $T^{\ast}(f)=\sum_{j=1}^n \lambda_j \overline{w(u_j)}K_{v_j}$ where $v_j=\varphi(u_j)$, there holds:
\begin{equation}\label{duo}\Vert T^{\ast}(f)\Vert^2\geq I_{v}^{-2} \sum_{j=1}^n |\lambda_j|^2|w(u_j)|^2 \Vert K_{v_j}\Vert^2.\end{equation}
Now, $I_u\leq I_v$ (if $h\in H^\infty$ interpolates $c$ at $v$, $h\circ \varphi$ interpolates $c$ at $u$), and clearly  $\Vert K_{v_j}\Vert^2\geq (1-|u_j|^2)\Vert K_{u_j}\Vert^2$ for all $j$. Lemma \ref{berber} ensues via (\ref{serge}).
 {\hfill$\square$}

\medskip

Finally, let $U\subset \D$ be a compact disk on which $w$ does not vanish, let $\delta=\inf_{U}|w|>0$ and $V$ be a closed disk with positive radius contained in $\varphi(U)$. Let $(v_j)$ be a sequence of $n$ equidistributed points on $\partial V$  and $(u_j)$ a sequence of length $n$ in $U$ such that $\varphi(u_j)=v_j$. We know (see \cite{Garnett:2007}, p.~284) that $I_v\leq C^n$ where $C$ only depends on $V$. We set $\eta^2=\hbox{\ dist} (U, \partial{\D})>0$. Lemma~\ref{berber} then gives us 
$$a_{n}(T) \gtrsim \delta\times  \eta\times C^{-2n},$$
which ends the proof of Thm.~\ref{saho}.
\end{proof}

\subsection{Maximal  special possible decay}
Theorem \ref{saho} proved that one never has superexponential decay for the approximation numbers $a_{n}(M_w C_\varphi)$. The following result indicates that, when $\Vert \varphi\Vert_\infty=1$, then whatever the non-zero weight $w$, the  numbers  $a_{n}(M_w C_\varphi)$ indeed have at most subexponential decay. \\
\begin{theorem}\label{subex} Suppose that the symbol $\varphi$  satisfies $\Vert \varphi\Vert_\infty=1$. Let now  $T=M_w\,C_\varphi$ where $w\in \mathcal{M}(H^2,\varphi)$, assumed to be compact. Then $\beta(T)=1$, i.e. there exists a sequence $(\eps_n)$ of positive numbers with limit $0$ such that 
 $$a_{n}(C_\varphi)\geq e^{-n\, \eps_n}.$$ 
\end{theorem}
\begin{proof}We borrow from \cite{LQR} some results on the Green capacity $\hbox{Cap}(X)$ of  Borel  subsets $X$ of $\D$. 
 \begin{itemize}
\item $X\subset Y\Rightarrow \hbox{Cap}(X)\leq \hbox{Cap}(Y)$.
\item $X_j\uparrow X\Rightarrow \hbox{Cap}(X_j)\uparrow \hbox{Cap}(X)$.
\item $\hbox{Cap}(X)=\hbox{Cap}(\partial{X})=\hbox{Cap}(\overline{X})$ when $X$ is connected and $\overline{X}\subset \D$.
\item When  the $X_j$'s are connected, $\hbox{Cap}(X_j)\to \infty$ if $\hbox{\ diam}\, X_j\to 1$, where $\hbox{\ diam}$ denotes  the diameter of $X_j$  for the pseudo-hyperbolic distance $\rho$ in $\D$:
$$\rho(a,b)=\Big\vert\frac{a-b}{1-\overline{a}b}\Big\vert\cdot$$
\end{itemize}
We also set $\Gamma(X)=\exp(-1/\hbox{Cap}(X))$.
With these notations, we will now prove the following extension of the main result of \cite{LQR} to weighted composition operators, which implies Theorem \ref{subex}.

\begin{theorem}\label{extension} Let $T=M_w\,C_\varphi$ where $\varphi$ is an \textnormal{arbitrary symbol}, and let $w\in \mathcal{M}(H^2,\varphi)$.  Then, $\beta(T)$ exists and moreover 
$$\beta(T)=\Gamma(\varphi(\D)).$$   In particular, $\Vert \varphi\Vert_\infty=1\Rightarrow \beta(T)=1$.
 
\end{theorem}
For the upper bound we can assume (see the proof of the lower bound), that  $\Vert \varphi\Vert_\infty<1$, so that $C_\varphi:H^2\to H^\infty$. Since $w\in \mathcal{M}(H^2,\varphi)$, we have $w\in H^2$ and $M_{w}:H^\infty\to H^2$ is bounded. By the ideal property of  approximation numbers, we obtain
\begin{equation}\label{sup}a_{n}(T)\leq \Vert M_{w}:H^\infty\to H^2\Vert\times a_{n}(C_{\varphi}:H^2\to H^\infty).\end{equation}

But for each $r$ with $\Vert \varphi\Vert_\infty<r<1$, we can write $\varphi=\beta_r\circ \gamma_r$ where $\beta_{r}(z)=rz$ and $\gamma_{r}(z)=\frac{\varphi(z)}{r}$, so that $C_\varphi=C_{\gamma_r}C_{\beta_r}$, with $C_{\beta_r}:H^2\to H^\infty$ and $C_{\gamma_r}:H^\infty\to H^\infty$, and (using again the ideal property of approximation numbers) $a_{n}(C_\varphi:H^2\to H^\infty)\leq \Vert C_{\beta_r}\Vert \times a_{n}(C_{\gamma_r}:H^\infty\to H^\infty)$.
Now, using a result of Widom (see \cite{LQR}, Theorem 3.6), we can assert that
\begin{equation}\label{wiwi}\limsup_{n\to \infty} \big[a_{n}(C_{\gamma_r}:H^\infty\to H^\infty)\big]^{1/n}\leq \Gamma(\gamma_{r}(\D))= \Gamma\Big(\frac{\varphi(\D)}{r}\Big).\end{equation}
 This implies:
$$\limsup_{n\to \infty} \big[a_{n}(C_\varphi:H^2\to H^\infty)\big]^{1/n}\leq \Gamma\Big(\frac{\varphi(\D)}{r}\Big).$$
Letting $r$ tend to $1^{-}$ gives $\limsup_{n\to \infty} \big[a_{n}(C_\varphi:H^2\to H^\infty)\big]^{1/n}\leq \Gamma(\varphi(\D))$.  Inserting this estimate in (\ref{sup}) finally gives 
$$\limsup_{n\to \infty} \big[a_{n}(T)\big]^{1/n}\leq \Gamma(\varphi(\D))$$
or else $\beta^{+}(T)\leq \Gamma(\varphi(\D))$. \\
For the lower bound, we will make use of a second  result of H.~Widom (see \cite{LQR}), in which $K$ is a compact subset of $\D$ of positive capacity, and $\|\cdot\|_{\mathcal{C}(K)}$ denotes the sup-norm on the space of continuous functions on $K$:\\
 If $E$ is a subspace of $H^2$ with $\dim E<n$, there exists $f\in B_{H^\infty}$, the unit ball of $H^\infty$, such that ($a$ denoting a positive absolute constant)
\begin{equation}\label{hawi}\Vert f-h\Vert_{\mathcal{C}(K)}\geq a[\Gamma(K)]^n \hbox{\ for  all }\ h\in E. \end{equation}
Now, since $w$ is not identically $0$, we can find a sequence $(r_j)$ with $r_j\uparrow 1$ such that $|z|=r_j\Rightarrow w(z)\neq 0$, implying 
$$\delta_j:=\inf_{|z|=r_j} |w(z)|>0.$$
Set $K_j=\varphi(r_j \T)$ and let $E\subset H^2$ with $\dim E<n$. By (\ref{hawi}), we can find $f\in B_{H^\infty}\subset B_{H^2}$ such that, for all $h\in E$: 
\begin{equation}\label{jeudi} \Vert f-h\Vert_{\mathcal{C}(K_j)}\geq a\big[\Gamma(K_j)\big]^n.\end{equation}
 It ensues that 
\begin{align*}
	\Vert w(f\circ \varphi-h\circ \varphi)\Vert_{\mathcal{C}(r_j\T)}
	&\geq 
	(\inf_{r_j \T} |w|)( \sup_{r_j \T} |f\circ \varphi-h\circ \varphi)|)
	\\
	&\geq
	\delta_j \Vert f-h\Vert_{\mathcal{C}(K_j)}
	\geq 
	a\delta_j [\Gamma(K_j)\big]^n.
\end{align*}

\noindent Moreover, we obviously have the inequality  $\Vert g\Vert_{\mathcal{C}(r_j\T)}\leq L_j \Vert g\Vert_2$  for all functions $g\in H^2$, where the  positive constant $L_j$ only depends  on $j$. This implies 
$$a\delta_j \big[\Gamma(K_j)\big]^n\leq \Vert w(f\circ \varphi-h\circ \varphi)\Vert_{\mathcal{C}(r_j\T)}\leq L_j \Vert w(f\circ \varphi-h\circ \varphi)\Vert_{2}=L_j \Vert Tf-Th\Vert_{2}$$ so that, for some positive constant $L'_j$ depending only on $j$:
$$d(Tf, TE)\geq L'_j \big[\Gamma(K_j)\big]^n.$$ Since this holds for every subspace $E$ of $H^2$ with $\dim E<n$, we derive from (\ref{alter}) that 
$a_{n}(T)\geq L'_j \big[\Gamma(K_j)\big]^n$. Taking $n$th-roots and passing to $\liminf$ as $n\to \infty$, we get 
\begin{equation}\label{seche}\beta^{-}(T)\geq \Gamma(K_j).\end{equation}
To finish, we set 
$$\omega_j =r_j\,\D,\quad  K'_j=\varphi(\overline{\omega_j})=\overline{\varphi(\omega_j)}\supset K_j.$$ 
Clearly, $\hbox{Cap}(K'_j)\geq \hbox{Cap}(K_j)$. However,  $\partial{\varphi(\omega_j)}\subset \varphi(\partial{\omega_j})=K_j$ since $\varphi(\omega_j)$ is open. We hence get, using the reminded results on the capacity of connected sets:
$$\hbox{Cap}(K'_j)=\hbox{Cap}[\varphi(\omega_j)]=\hbox{Cap}[\partial{\varphi(\omega_j)}]\leq \hbox{Cap}(K_j).$$
So that $\hbox{Cap}(K'_j)=\hbox{Cap}(K_j)$ and that, using (\ref{seche}):
\begin{equation}\label{secher}\beta^{-}(T)\geq \Gamma(K'_j).\end{equation}
Letting $j\to \infty$, we obtain, since $K'_j\uparrow \varphi(\D)$:
$\beta^{-}(T)\geq \Gamma(\varphi(\D))$ and hence $\beta(T)=\Gamma(\varphi(\D))$.\\
 Now, suppose that $\Vert \varphi\Vert_\infty=1$. By conformal invariance of the parameters involved 
 ($\beta(C_\varphi), \hbox{\ Cap}(\varphi(\D)), \hbox{\ diam}(\varphi(\D))$), we can assume that $\varphi(0)=0$, without loss of generality. In that case, $\hbox{\ diam}(K'_j)\to 1$ and subsequently $\Gamma(K'_j)\to 1$. So that $\Gamma(\varphi(\D))=1$, and   
 finally $\beta(T)=1$.\end{proof}

\noindent{\bf Remark.} In the preceding, we limited ourselves to the case of the Hardy space, but several results, for example Theorems \ref{saho} and \ref{subex}, hold true for other Hilbert spaces of analytic functions on the disk.  Indeed, they hold true \cite{LQR} if the ambient norm in $H$ is defined by 
$$\Vert f\Vert^2=|f(0)|^2+\int_{\D}|f'(z)|^2 \omega(z)\frac{dA(z)}{\pi}$$
where $A$ is the area measure on $\C$ and $\omega$  a radial weight on $\D$, integrable on $(0,1)$. This framework includes for example the Hardy space ($\omega(r)=1-r$), the Bergman space ($\omega(r)=(1-r)^2$),  or the Dirichlet space $\mathcal{D}$ ($\omega(r)=1$), even though the multiplier space $\mathcal{M}(\mathcal{D})$ is not $H^\infty$.
To prove these results, we can adapt the methods of \cite{LQR}. We skip the details.\\

\section{Modular theory and Hardy spaces on strips}\label{Section:modular}

\subsection{From von Neumann algebras to Hardy spaces}

We now begin our discussion about how Hardy spaces and composition operators arise in the context of (specific) inclusions of von Neumann algebras and their real standard subspaces. 

The following definition is motivated by algebraic quantum field theory \cite{Haag:1996} (in two dimensions), where one builds models by assigning von Neumann algebras to regions (subsets) in $\Rl^2$ such that a number of geometric properties (inclusions of subsets, causal separation w.r.t. the Minkowski inner product, symmetries) are carried into corresponding algebraic properties (inclusions of algebras, commutants, group actions). 

A region of particular interest is the {\em wedge} $W:=\{x\in\Rl^2:\pm x_\pm>0\}$ (here $x=(x_+,x_-)$ is the parameterization of $x\in\Rl^2$ in light cone coordinates) \cite{Borchers:1992,BaumgrtelWollenberg:1992,BuchholzLechnerSummers:2011}. The following definition, taken from \cite{BuchholzLechnerSummers:2011}, lists the essential properties of a von Neumann algebra $\NN$ (together with the space-time translations $U$ and a vacuum vector $\Om$) that are required for $\NN$ to resemble the localization region $W$ in a physically reasonable manner.

\begin{definition}\label{def:BT}
     A Borchers triple consists of a von Neumann algebra $\NN$, acting on some Hilbert space $\Hil$, a cyclic and separating unit vector $\Om\in\Hil$, and a unitary strongly continuous representation $U$ of $\Rl^2$ such that 
     \begin{enumerate}[i)]
	  \item $U$ has {\em positive energy}. That is, writing $U(x)=e^{ix_+P_+}e^{ix_-P_-}$, the two generators are positive, $P_\pm>0$.
	  \item $U(x)\Om=\Om$ for all $x\in\Rl^2$.
	  \item $U(x)\NN U(x)^{-1}\subset\NN$ for $x_+>0$, $x_-<0$.
     \end{enumerate}
\end{definition}

We will here be mostly interested with certain modular data derived from a Borchers triple, and now recall the relevant notions (see \cite{Longo:2008} for details and proofs of the claims made here).

Given a Borchers triple, we consider the closed real subspace
\begin{align}
	K:=\{N\Om\,:\,N=N^*\in\NN\}^{\|\cdot\|}\subset\Hil
\end{align}
which is {\em standard} in the sense that $K+iK$ is dense in $\Hil$, and $K\cap iK=\{0\}$. To any such standard subspace one can associate a Tomita operator, that is the antilinear involution defined as 
\begin{align}
	S:K+iK\to K+iK\,,\qquad k_1+ik_2\mapsto k_1-ik_2\,.
\end{align}
This operator is densely defined, closed, and typically unbounded. Its polar decomposition $S=J\Delta^{1/2}$ gives rise to a positive non-singular linear operator $\Delta^{1/2}$ with domain $\dom\Delta^{1/2}=K+iK$ (the modular operator) and an anti unitary involution $J$ (the modular conjugation), satisfying $J\Delta^{it}=\Delta^{it}J$, $t\in\Rl$, and $J\Om=\Om$, $\Delta^{1/2}\Om=\Om$.

\medskip

The Hardy space (on a strip) that will arise later will be associated with the domain of the modular operator $\Delta^{1/2}$. As a first step towards this link, let us equip the dense subspace $\dom\Delta^{1/2}\subset\Hil$ with the graph norm, defined by the scalar product
\begin{align}\label{eq:graph-sp}
	\langle\psi,\varphi\rangle_\Delta
	:=
	\langle\psi,\varphi\rangle+\langle \Delta^{1/2}\psi,\Delta^{1/2}\varphi\rangle
	=
	\langle\psi,(1+\Delta)\varphi\rangle
	\,.
\end{align}
Since $\Delta^{1/2}$ is a closed operator, $\dom\Delta^{1/2}$ is closed in the graph norm \cite{ReedSimon:1972} $\|\cdot\|_\Delta=\langle\cdot,\cdot\rangle_\Delta^{1/2}$. That is, $(\dom\Delta^{1/2},\langle\cdot,\cdot\rangle_\Delta)$ is a complex Hilbert space. When considering $\dom\Delta^{1/2}$ with this scalar product, we refer to it as $\Dc$. 

The following proposition gathers basic information on the modular data from this point of view. 

\begin{proposition}\label{lemma:D}
Let $K$ be a closed real standard subspace and $\Dc$ the complex Hilbert space defined as $\dom\Delta^{1/2}$ with its graph scalar product \eqref{eq:graph-sp}.
     \begin{enumerate}[i)]
	  \item On $\Dc$, the Tomita operator $S$ is an anti unitary involution.
	  \item On $\Dc$, the modular unitaries $\Delta^{it}$, $t\in\Rl$, are still unitary.
	  \item For $0\leq\mu\leq\frac{1}{2}$, the modular operator $\Delta^{\mu}$ is a linear bounded operator of norm at most one when viewed as a map from $\Dc$ to $\Hil$.
	  \item The modular conjugation $J$ is an antiunitary operator $\Dc_\Delta\to\Dc_{\Delta^{-1}}$.
     \end{enumerate}
\end{proposition}
\begin{proof}
     {\em i)} Let $\psi,\varphi\in\Dc$. Then, since $S^*S=\Delta$ and $\Delta^{1/2}S=\Delta^{1/2}J\Delta^{1/2}=J$, 
     \begin{align*}
	  \langle S\psi,S\varphi\rangle_\Delta
	  &=
	  \langle S\psi,S\varphi\rangle+\langle \Delta^{1/2}S\psi,\Delta^{1/2}S\varphi\rangle
	  =
	  \overline{\langle \psi,\Delta\varphi\rangle}+\langle J\psi,J\varphi\rangle
	  \\
	  &=
	  \overline{\langle\Delta^{1/2}\psi,\Delta^{1/2}\varphi\rangle+\langle\psi,\varphi\rangle}
	  =
	  \overline{\langle\psi,\varphi\rangle_\Delta}\,.
     \end{align*}
     This shows that $S$ is anti unitary on $\Dc$.
     
     \noindent{\em ii)} This is clear because $\Delta^{it}$ commutes with $\Delta^{1/2}$.
     
     \noindent{\em iii)} For $\psi\in\Dc$, the $\Hil$-valued function $z\mapsto\Delta^{-iz}\psi$ is holomorphic on the strip $0<\im(z)<\frac{1}{2}$, continuous on its closure, and norm-constant in the real direction, $\|\Delta^{-i(x+iy)}\psi\|=\|\Delta^y\psi\|$. Thus the three lines theorem \cite{Conway:1978} implies, $0\leq y\leq\frac{1}{2}$,
     \begin{align*}
	  \|\Delta^y\psi\|^{1/2}\leq\|\psi\|^{1/2-y}\cdot\|\Delta^{1/2}\psi\|^y\,,
     \end{align*}
     from which we read off
     \begin{align*}
	  \|\Delta^y\psi\|
	  \leq
	  \|\psi\|^{1-2y}\cdot\|\Delta^{1/2}\psi\|^{2y}
	  \leq
	  \sqrt{\|\psi\|^2+\|\Delta^{1/2}\psi\|^2}
	  =
	  \|\psi\|_\Delta\,.
     \end{align*}
     \noindent{\em iv)} This follows from, $\psi,\varphi\in\dom(\Delta^{1/2})$,
     \begin{align*}
	  \langle J\psi,J\varphi\rangle_\Delta
	  &=
	  \langle J\psi,J\varphi\rangle+\langle\Delta^{1/2}J\psi,\Delta^{1/2}J\varphi\rangle  
	  \\
	  &=
	  \overline{\langle \psi,\varphi\rangle}+\overline{\langle\Delta^{-1/2}\psi,\Delta^{-1/2}\varphi\rangle}
	  =
	  \overline{\langle J\psi,J\varphi\rangle_{\Delta^{-1}}}
	  \,.
     \end{align*}
\end{proof}

To draw the connection to Hardy spaces, we recall a theorem of Borchers~\cite{Borchers:1992} on the commutation relation of the modular unitaries $\Delta^{it}$, $t\in\Rl$, and the translation unitaries $U(x)$, $x\in\Rl^2$: If $(\NN,U,\Om)$ is a Borchers triple, then there holds
\begin{align}\label{eq:BorchersCommutationRelations}
	\Delta^{it}U(x)\Delta^{-it}&=U(\La(t)x)\,,\qquad
	JU(x)J=U(-x)\,,\qquad x\in\Rl^2\,,
\end{align}
where $(\La(t)x)_\pm=e^{\mp 2\pi t}x_\pm$.

The commutation relations \eqref{eq:BorchersCommutationRelations} imply that in the presence of a Borchers triple, the operators $U(x)$, $\Delta^{it}$, and $J$ generate a (anti-)unitary strongly continuous representation of the proper Poincar\'e group $\PG_+=\SO(1,1)\rtimes\Rl^2$. We will denote this extended representation by the same letter~$U$.

As the basic building blocks, we are interested in the irreducible subrepresentations of~$U$. We call a representation degenerate if there exists a non-zero vector~$\Psi$ which is invariant under $U(x)$ for all $x\in\Rl^2$, i.e. $\Psi\in\ker P_+\cap\ker P_-$. In physical models, this vector can usually only be a multiple of $\Om$, so that we may restrict to non-degenerate representations of $\PG_+$.

The non-degenerate, irreducible, unitary, strongly continuous positive energy representations of $\PG_+$ can be classified up to unitary equivalence according to the joint spectrum of the generators $P_\pm$, denoted $\Sp(U|_{\Rl^2})$. There are equivalence classes of three types:
\begin{enumerate}
	\item[$m$)] $\Sp(U|_{\Rl^2})=\{p\in\Rl^2\,:\,p_\pm>0,\;\;p_+\cdot p_-=m^2\}$ for some $m>0$,
	\item[$0+$)] $\Sp(U|_{\Rl^2})=\{p\in\Rl^2\,:\,p_+\geq0,\,p_-=0\}$,
	\item[$0-$)] $\Sp(U|_{\Rl^2})=\{p\in\Rl^2\,:\,p_-\geq0,\,p_+=0\}$.
\end{enumerate}
The parameter $m$ has the physical interpretation of a mass. We therefore refer to representations of type $m>0$ as ``massive'', and to representations of type $0\pm$ as ``massless''. For a concise notation, we will adopt the convention to label objects by a single label $m$, which can either take positive values, referring to type $m)$, or the two special values $m=0\pm$, referring to type~$0\pm)$.

\medskip

The irreducible representation $U_m$ of type $m$ can be conveniently realized on the Hilbert space $\Hil=L^2(\Rl,d\te)$. In fact, we have for all types the same modular unitaries and conjugation \cite{LechnerLongo:2014}
\begin{align}\label{eq:mod-irrep}
     \Hil=L^2(\Rl,d\te)
     \,,\qquad
     (\Delta^{it}\psi)(\te)&= \psi(\te-2\pi t),
     \qquad
     (J\psi)(\te)=\overline{\psi(\te)}
     \,.
\end{align}
The translation operators are multiplication operators depending on the type, namely
\begin{align}\label{eq:rep-weights}
     (U_m(x)\psi)(\te) 
     &= 
     w_{m,x}(\te)\cdot\psi(\te)\,,
     \\
     w_{0\pm,x}(\te) &= e^{ix_\pm \,e^{\pm\te}}\,,\quad
     w_{m,x}(\te)= e^{im(x_+e^\te+x_-e^{-\te})}\,.
\end{align}
The functions $w_{m,x}$ will later serve as the weights of our weighted composition operators.

Considering the representation \eqref{eq:mod-irrep}, it becomes clear that the domain $\Dc\subset L^2(\Rl)$ of $\Delta^{1/2}$ consists of functions that have an analytic continuation to the strip region $\Strip_{0,\pi}$, a special case of the more general strip
\begin{align}\label{eq:Strip-pi}
     \Strip_{a,b}:=\{\zeta\in\Cl\,:\,a<\im\zeta<b\}\,,\qquad a<b\,,
\end{align}
and satisfy certain bounds on this strip. Before we make this precise, let us recall some properties of functions analytic in a strip, and corresponding function spaces.

\bigskip

Given $f\in\Hol(\Strip_{a,b})$ (the holomorphic functions $\Strip_{a,b}\to\Cl$), we write $f_\la$, $a<\la<b$, for its restriction to the line $\Rl+i\la$. Denoting the usual norm of $L^2(\Rl)$ by $\|\cdot\|_2$, we consider the norm
\begin{align}\label{eq:Hardy-Banach-Norm}
     \bno{f}_{\Strip_{a,b}}
     :=
     \sup_{a<\la<b}\|f_\la\|_2\in[0,+\infty]
\end{align}
and set 
\begin{align}
     H^2_B(\Strip_{a,b}):=\{f\in\Hol(\Strip_{a,b})\,:\,\bno{f}_{\Strip_{a,b}}<\infty\}\,.
\end{align}
We recall the following facts \cite{SteinWeiss:1971}:
\begin{enumerate}[{\em i)}]
     \item $(H^2_B(\Strip_{a,b}),\bno{\cdot}_{\Strip_{a,b}})$ is a Banach space.
     \item Any $f\in H^2_B(\Strip_{a,b})$ has $L^2$-boundary values on the two boundaries $\Rl+ia$ and $\Rl+ib$, i.e. $f_{a+\eps}$ and $f_{b-\eps}$ converge in $L^2(\Rl)$ as $\eps\searrow 0$. By a slight abuse of notation, we will denote these boundary values as $f_a$, $f_b\in L^2(\Rl)$. 
     \item As an expression of the maximum principle, the function $(a,b)\ni\la\mapsto\|f_\la\|_2$ is logarithmically convex for $f\in H^2_B(\Strip)$. In particular, $\bno{f}=\max\{\|f_a\|_2,\|f_b\|_2\}$.
\end{enumerate}

We will refer to $H_B^2(\Strip_{a,b})$ as {\em Hardy Banach space} to distinguish it from a Hardy Hilbert space on $\Strip_{a,b}$ to be introduced next.

Indeed, as the strip is an unbounded region, there exist two different types of Hardy Hilbert spaces for this domain: The conformally invariant Hardy space, defined in terms of harmonic majorants, and the not conformally invariant Hardy space, defined in terms of $L^2$-integrals over a sequence of Jordan curves tending to the boundary of the strip \cite[Ch.~10]{Duren:1970}. 

For our purposes, only the latter space will be relevant. It will be convenient to characterize it in terms of a Riemann map $\tau:\Disc\to\Strip$. (To lighten our notation, we write $\Strip$ instead of $\Strip_{a,b}$ when the boundaries of the strip are arbitrary.) Namely, we define
\begin{align}
	H^2(\Strip)
	:=
	\{f\in\Hol(\Strip)\,:\,\sqrt{\tau'}\cdot(f\circ\tau)\in H^2(\Disc)\}
	\,.
\end{align}
This is a Hilbert space with scalar product
\begin{align}
     \langle f,g\rangle_\Strip := \langle \sqrt{\tau'}\cdot(f\circ\tau),\,\sqrt{\tau'}\cdot(g\circ\tau)\rangle_\Disc
     \,,
\end{align}
and $(H^2(\Strip),\langle\cdot,\cdot\rangle_\Disc)$ depends on the choice of $\tau$ only up to changing the norm $\|f\|_\Strip:=\langle f,f\rangle_\Strip^{1/2}$ to an equivalent Hilbert norm. We may therefore fix $\tau$, and make the choice
\begin{align}
     \tau:\Disc\to\Strip_{a,b}\,,\quad \tau(z):=\frac{2(b-a)}{\pi}{\rm arctanh}(z)+\frac{i}{2}(a+b)\,.
\end{align}
This is a biholomorphic mapping $\tau:\Disc\to\Strip_{a,b}$, and elementary calculations show that it has inverse and derivative, $\zeta\in\Strip_{a,b}$, $z\in\Disc$,
\begin{align}\label{eq:tau-formulas}
     \tau^{-1}(\zeta)
     &=
     \tanh\left(\frac{\pi}{2(b-a)}\,\zeta-\frac{i\pi}{4}\frac{b+a}{b-a}\right)
     \,,\quad 
     \tau'(z)=\frac{2(b-a)}{\pi}\frac{1}{1-z^2}\,,
     \\
     &(\tau^{-1})'(\zeta)
     =
     \frac{\pi}{2(b-a)}\frac{1}{\cosh^2\left(\frac{\pi}{2(b-a)}\,\zeta-\frac{i\pi}{4}\frac{b+a}{b-a}\right)}
     \,.
     \label{eq:derivative-tau-inverse}
\end{align}

\newpage

\begin{proposition}\label{prop1}
     \begin{enumerate}[i)]
	  \item $H_B^2(\Strip)$ and $H^2(\Strip)$ coincide as linear spaces.
	  \item The two norms $\bno{\cdot}$ and $\|\cdot\|_{\Strip_{a,b}}$ are equivalent: For any $f\in\Hol(\Strip)$, there holds
	  \begin{align}
	       \frac{1}{\sqrt{2\pi}}\bno{f} \leq \|f\|_\Strip \leq \frac{1}{\sqrt{\pi}}\bno{f}\,.
	  \end{align}
	  \item The scalar product of $H^2(\Strip_{a,b})$ can be written as
	  \begin{align}\label{eq:strip-scalar-product}
	       \langle f,g\rangle_{\Strip_{a,b}}
	       =
	       \frac{1}{2\pi}\Big(
	       \langle f_a,g_a\rangle_2+\langle f_b,g_b\rangle_2
	       \Big)
	       \,,\qquad f,g\in H^2(\Strip_{a,b}).
	  \end{align}
     \end{enumerate}
\end{proposition}
\begin{proof}
     We first work on the special strip $\Strip$ given by $a=-1$, $b=1$, and introduce for $f\in\Hol(\Strip)$ the notation $s(f):=\frac{1}{2}\sup_{0\leq y<1}(\|f_y\|_2^2+\|f_{-y}\|_2^2)\in[0,+\infty]$.
     
     It was shown in \cite[Thm.~2.1 \& 2.2]{BakanKaijser:2007} that for $f\in\Hol(\Strip)$, one has $f\circ\tau\in H^2(\Disc)$ if and only if $s(w\cdot f)<\infty$, and in this case, $\|f\circ\tau\|_\Disc^2=s(w\cdot f)$, with the weight $w(\zeta)=(2\cosh\frac{\pi\zeta}{4})^{-1}$. 
     
     The condition that some $f\in\Hol(\Strip)$ lies in $H^2(\Strip)$, i.e. that $\sqrt{\tau'}(f\circ\tau)=(f/\sqrt{(\tau^{-1})'})\circ\tau\in H^2(\Disc)$, is therefore equivalent to $s(w/\sqrt{(\tau^{-1})'}\cdot f)<\infty$. But in view of \eqref{eq:derivative-tau-inverse}, $w(\zeta)/\sqrt{(\tau^{-1})'(\zeta)}=\pi^{-1/2}$. We thus have that $f\in\Hol(\Strip)$ lies in $H^2(\Strip)$ if and only if $s(f)<\infty$, and in this case, 
     \begin{align}\label{eq:norm-sup}
	  \|f\|_\Strip^2
	  =
	  \|\sqrt{\tau'}(f\circ\tau)\|_\Disc^2
	  =
	  \frac{1}{\pi}\,s(f)
	  =
	  \frac{1}{2\pi}\sup_{0\leq y<1}\left(\|f_y\|_2^2+\|f_{-y}\|_2^2\right)\,.
     \end{align}
     As the supremum on the right hand side clearly lies between $\bno{f}^2$ and $2\bno{f}^2$, the claimed equivalence of norms in $ii)$ follows. This also implies $i)$. 
     
     To establish $iii)$, we use that $(-1,1)\ni y\mapsto\|f_y\|_2$ is logarithmically convex for $f\in H_B^2(\Strip)$. Thus $y\mapsto\|f_y\|_2^2+\|f_{-y}\|_2^2$ is convex, which implies that the supremum in \eqref{eq:norm-sup} is taken for the boundary values at $y=1$, i.e. 
     \begin{align*}
	  \|f\|_\Strip^2
	  =
	  \frac{1}{2\pi} \left(\|f_1\|_2^2+\|f_{-1}\|_2^2\right)
	  =
	  \frac{1}{2\pi}\Big(\langle f_1,f_1\rangle_2+\langle f_{-1},f_{-1}\rangle_2\Big)
	  \,.
     \end{align*}
     This implies {\em iii)}.
     
     It remains to proceed from $\Strip_{-1,1}$ to a general strip $\Strip_{a,b}$ by means of the variable transformation $\varphi:\Strip_{-1,1}\to\Strip_{a,b}$, $\varphi(\zeta):=\frac{b-a}{2}\zeta+\frac{i}{2}(a+b)$. But by elementary substitutions, one finds that $f\mapsto f\circ\varphi$ is a bijection $H_B^2(\Strip_{a,b})\to H_B^2(\Strip_{-1,1})$, with $\|(f\circ\varphi)_y\|_2^2=\frac{2}{b-a}\|f_{(b-a)y/2}\|_2^2$. Since also $\|f\circ\varphi\|_{\Strip_{-1,1}}^2=\frac{2}{b-a}\|f\|_{\Strip_{a,b}}^2$, the properties {\em i)}--{\em iii)} follow from the special case $a=-1$, $b=1$.
\end{proof}

Given the action of the unitaries $\Delta^{it}$ \eqref{eq:mod-irrep} in the irreducible representations $U_m$ arising from the modular data of our Borchers triple, the scalar product \eqref{eq:strip-scalar-product} of $H^2(0,\pi)$ is strongly reminiscent of the graph scalar product \eqref{eq:graph-sp}. To make this match exact, we will use a rescaled version of the graph scalar product, namely
\begin{align}\label{eq:graph-sp2}
     \langle\psi,\varphi\rangle_\Delta
     :=
     \frac{1}{2\pi}\left(\langle\psi,\varphi\rangle+\langle \Delta^{1/2}\psi,\Delta^{1/2}\varphi\rangle\right)
     \,.
\end{align}
Clearly, Proposition~\ref{lemma:D} still holds with this equivalent scalar product. Moreover, we have the following concrete realization of $\Dc$.

\begin{proposition}
     Consider the modular data \eqref{eq:mod-irrep}, and denote by $\Dc$ the complex Hilbert space $\dom\Delta^{1/2}$ with scalar product \eqref{eq:graph-sp2}. 
     \begin{enumerate}[i)]
	  \item $\Dc=H^2(\Strip_{0,\pi})$ as complex Hilbert spaces.
	  \item The Tomita operator $S$ acts on $H^2(\Strip_{0,\pi})$ by ``crossing symmetry'', i.e. 
	  \begin{align}\label{eq:tomita-hardy}
	       (S\psi)(\zeta)
	       =
	       \overline{\psi(i\pi+\bar\zeta)}\,,\qquad \zeta\in\Strip_{0,\pi}\,.
	  \end{align}
     \end{enumerate}     
\end{proposition}
\begin{proof}
     It was shown in \cite[Lemma~A.1]{LechnerLongo:2014} that the real standard subspace $K=\ker(1-J\Delta^{1/2})$ is given by 
     \begin{align*}
	  K=\{\psi\in H_B^2(\Strip_{0,\pi})\,:\,\overline{\psi_{\pi}(\te)}=\psi_0(\te)\;\,a.e.\}\,.
     \end{align*}
     In view of the analyticity properties of $\psi$, this implies that $K$ consists exactly of those functions $\psi\in H_B^2(\Strip_{0,\pi})$ that satisfy
     \begin{align}\label{eq:char-H}
	  \overline{\psi(i\pi+\overline{\zeta})}=\psi(\zeta)\,,\qquad \zeta\in\Strip_{0,\pi}\,.
     \end{align}
     Clearly any $f\in H_B^2(\Strip_{0,\pi})$ can be written as $f=\psi+i\varphi$ with $\psi,\varphi\in K$, so that we see $H_B^2(\Strip_{0,\pi})=K+iK$. But as $S$ is an antilinear involution, with domain $\Dc=K+iK$, and $H_B^2(\Strip_{0,\pi})=H^2(\Strip_{0,\pi})$, it follows that $\Dc=H^2(\Strip_{0,\pi})$ as linear spaces. Also the graph scalar product \eqref{eq:graph-sp2} coincides with the scalar product of $H^2(\Strip_{0,\pi})$ by Prop.~\ref{prop1}~{\em iii)}. This shows {\em i)}.
     
     The Tomita operator $S$ is uniquely fixed by being an antilinear involution and $Sk=k$ for all $k\in K$. But in view of the characterization \eqref{eq:char-H} of $K$, it is clear that the antilinear involution defined in \eqref{eq:tomita-hardy} leaves $K$ pointwise invariant. This shows {\em ii)}.
\end{proof}

\bigskip

Having established the connection between modular data and Hardy spaces, we now explain how composition operators appear in this setting. 

Any Borchers triple defines a quantum field theory on $\Rl^2$ \cite{BuchholzLechnerSummers:2011}, which makes this concept interesting in the context of constructing models. However, the quantum field theories arising from Borchers triples might be pathological in the sense of containing no strictly local observables, a situation that arises when the inclusions $U(x)\NN U(x)^{-1}\subset\NN$, $x\in W$, have trivial relative commutants. These pathological situations can however be ruled out \cite{BuchholzLechner:2004,Lechner:2008} when the so-called {\em modular nuclearity condition} \cite{BuchholzDAntoniLongo:1990-1,BuchholzDAntoniLongo:1990} holds. This condition requires that the maps
\begin{align}\label{eq:Xi}
     \Xi_{x,\mu}:\NN\to\Hil\,,\qquad \Xi_{x,\mu}N:=\Delta^\mu U(x)N\Om\,,\qquad x\in W,\;0<\mu<\frac{1}{2}\,,
\end{align}
are {\em nuclear}\footnote{The condition that a linear map $X$ between two Banach spaces is nuclear is slightly weaker than $X$ having summable approximation numbers \cite{Pietsch:1972}.} as linear maps between the two Banach spaces $(\NN,\|\cdot\|_{\B(\Hil)})$ and $\Hil$. Whereas $\Xi_{x,\mu}$ is always bounded by modular theory, it is in general not compact, so that nuclearity of \eqref{eq:Xi} is a non-trivial requirement.

\bigskip

To investigate the approximation numbers of $\Xi_{x,\mu}$, we split this map as
\begin{align}\label{eq:splitting}
     \Xi_{x,\mu}:\NN\stackrel{Y}\longrightarrow\Dc\stackrel{U(x)}\longrightarrow\Dc\stackrel{\Delta^\mu}\longrightarrow\Hil
     \,,
\end{align}
where the first operator, defined as $Y(N):=N\Om$, is bounded: For any $N\in\NN$, we have 
\begin{align*}
     \|Y(N)\|_\Delta^2
     &=
     \|N\Om\|^2+\langle \Delta^{1/2}N\Om,\Delta^{1/2} N\Om\rangle
     =
     \|N\Om\|^2+\langle JN^*\Om,JN^*\Om\rangle
     \\
     &\leq
     2\|N\|^2,
\end{align*}
because $\|J\|=1$ and $\|N^*\|=\|N\|$. 

The last operator in \eqref{eq:splitting}, $\Delta^\mu$, is bounded as an operator $\Dc\to\Hil$ (see Prop.~\ref{lemma:D}~{\em iii)}).

\begin{lemma}\label{lemma:Ux}
     Let $(\NN,U,\Om)$ be a Borchers triple. Then the translations $U(x)$, $x\in W$, are isometries as maps on the Hilbert space $\Dc$.
\end{lemma}
\begin{proof}
     It is known that the commutation relations \eqref{eq:BorchersCommutationRelations} imply that the operator $\Delta^{1/2}U(x)\Delta^{-1/2}$, $x\in W$, is defined on $\dom\Delta^{-1/2}$, and coincides there with $JU(x)J=U(-x)$ \cite[Thm. 2.3.1 f)]{Longo:2008}. We therefore find for $\psi\in\Dc$ and $x\in W$ the equation $\Delta^{1/2}U(x)\psi=\Delta^{1/2}U(x)\Delta^{-1/2}\Delta^{1/2}\psi=U(-x)\Delta^{1/2}\psi$, and consequently
     \begin{align*}
	  \|U(x)\psi\|_\Delta^2
	  &=\|U(x)\psi\|^2+\|\Delta^{1/2}U(x)\psi\|^2
	  \\
	  &=\|U(x)\psi\|^2+\|U(-x)\Delta^{1/2}\psi\|^2
	  =\|\psi\|_\Delta^2\,,
     \end{align*}
     where we have used that $U(x)$ is unitary on $\Hil$. This shows that $U(x)$ is an isometry on $\Dc$. Note that $U(x)$ is (except for trivial cases) not unitary because it does not have full range.
\end{proof}

The product of the last two operators in the split \eqref{eq:splitting},
\begin{align}\label{eq:Dxmu-ops}
     D_{x,\mu}:\Dc\to\Hil\,,\qquad D_{x,\mu}=\Delta^\mu U(x)\,,
\end{align}
can however be compact (and even have approximation number that go to zero quite fast), analogously to the situation encountered in Thm~\ref{thm:disc-in-disc}.

\medskip

To obtain estimates on the approximation numbers $a_n(D_{x,\mu})$, one splits $\Hil$ and $\Dc$ into irreducible subspaces of $U$. Then each subspace takes the form $\Hil=L^2(\Rl,d\te)$, $\Dc=H^2(\Strip_{0,\pi})$, and $U_m(x):\DD\to\DD$ acts by multiplication with the (analytic continuation of the) weight $w_{m,x}$ \eqref{eq:rep-weights}, depending on the representation type $m$. (Note that the analytically continued weight functions $w_{m,x}\in H^\infty(\Strip_{0,\pi})$ \eqref{eq:rep-weights} are bounded and inner, for any $m$ and $x$.)

Explicitly, the operator $D_{x,\mu}^{(m)}=D_{x,\mu}$ then takes the concrete form of a ``weighted restriction operator'', $0<\mu<\pi$, $x\in W$,
\begin{align}\label{eq:Dxmmu}
     D^{(m)}_{x,\mu}:H^2(\Strip_{0,\pi})&\to L^2(\Rl)\,,\\
     (D^{(m)}_{x,\mu}\psi)(\te)
     &:=
     w_{m,x}(\te+i\mu)\cdot \psi(\te+i\mu)
     \,.
\end{align}
This observation warrants a more systematic analysis of weighted restriction operators on Hardy spaces on strips. Before we enter into this analysis in the next section, let us comment on the relation between the operators $D^{(m)}_{x,\mu}$ and $\Xi_{x,\mu}$ \eqref{eq:Xi}.

Estimates on the approximation numbers of $D^{(m)}_{x,\mu}$ do not imply corresponding estimates on the maps $\Xi_{x,\mu}$ \eqref{eq:Xi}. To establish bounds on the $a_n(\Xi_{x,\mu})$, one also has to take into account the multiplicities occurring in the decomposition of $U$ into irreducibles. But for typical examples of Borchers triples, this analysis involves basically only (symmetrized) tensor powers of operators of the form $D^{(m)}_{x,\mu}$ \cite{AlazzawiLechner:2016}. For this reason, it is of interest to determine strong decay properties of the approximation numbers of $D^{(m)}_{x,\mu}$.

\subsection{Weighted restriction operators on $H^2(\Strip)$}

As a slight generalization of what appeared before, we consider here the following setting: Let~$\Strip\subset\Cl$ be a strip domain, which will be fixed in the following. To define the operators we want to study, we take a narrower strip $\tilde\Strip\subset\Strip$, such that the closure of the smaller strip is contained in the larger one, and a weight function $w\in H^\infty(\Strip)$. We are then interested in the mappings $f\mapsto (w\cdot f)|_{\tilde\Strip}$, considered as operators $H^2(\Strip)\to H^2(\tilde\Strip)$.

For simplicity, we will always assume that both $\Strip$ and $\tilde\Strip$ are symmetric around the real axis, i.e. $\Strip=\Strip_{-b,b}$ for some $b>0$ and $\tilde\Strip=\la\Strip$ for some $0<\la<1$. We then define
\begin{align}
     R_{w,\la}:H^2(\Strip)\to H^2(\la\Strip)\,,\qquad R_{w,\la}f:=(w\cdot f)|_{\la\Strip}\,.
\end{align}
As a limiting case as $\la\to0$, we also define
\begin{align}\label{eq:Rw0}
     R_{w,0}:H^2(\Strip)\to L^2(\Rl)\,,\qquad R_{w,0}:=(w\cdot f)_0
\end{align}
as the restriction of $wf$ to the real line. It is clear from Prop.~\ref{prop1}~$i)$ and the form of the norm \eqref{eq:Hardy-Banach-Norm} that the restriction maps $R_{1,\la}$, $0\leq\la<1$, with trivial weight $w=1$ are bounded. Since restriction of $f\in H^2(\Strip)$ to $\la'\Strip$ is the same as first restricting $f$ to $\la\Strip$, $\la>\la'$, and then to $\la'\Strip$, we find that there is a constant $c$ such that
\begin{align}\label{eq:anscale}
	a_n(R_{w,\la'})\leq c\,a_n(R_{w,\la})\,,\qquad 0\leq\la'\leq\la\leq1,\;n\in\Nl\,.
\end{align}
In particular, the operators $R_{w,0}$ mapping to $L^2(\Rl)$ \eqref{eq:Rw0} can be estimated in terms of the $R_{w,\la}$, $\la>0$. The latter operators map between Hardy spaces and can be reformulated as composition operators as follows.

Let $0<\la<1$ and
\begin{align}
     L_\la:H^2(\la\Strip)\to H^2(\Strip)\,,\qquad (L_\la f)(z):=\sqrt{\la}f(\la z)\,.
\end{align}
Taking into account that the Riemann maps $\tau_\la$ for $\la\Strip$ and $\tau$ for $\Strip$ are related by $\tau_\la=\la\tau$, it follows that $L_\la$ is unitary. Furthermore, the product $\la^{-1/2}\,L_\la R_{1,\la}$ is easily seen to be the composition operator $C_\la$ on $H^2(\Strip)$ with linear symbol $z\mapsto\la z$. This shows that the restriction operators $R_{1,\la}$ are unitarily similar to composition operators. Furthermore, we note that $C_\la$ (and thus $R_{1,\la}$) are {\em not} compact. This is so because on $H^2(\Strip)$ there exist no compact composition operators at all \cite{ShapiroSmith:2003}. The {\em weighted} operators $R_{w,\la}$ can however be compact, depending on the weight $w$, similar to the example in Thm.~\ref{thm:disc-in-disc}.

\medskip

For the following analysis of the approximation numbers of $R_{w,\la}$, we recall that $H^2(\Strip)$ is a reproducing kernel Hilbert space. Its kernel function $K_\Strip$ is related to the well-known Szeg\" o kernel \cite{Shapiro:1993} $K_\Disc(z,z')=(1-\overline{z}z')^{-1}$ of $H^2(\Disc)$ by 
\begin{align}
     K_\Strip(\zeta,\zeta')
     &=
     \sqrt{\overline{(\tau^{-1})'(\zeta)}}\cdot K_\Disc(\tau^{-1}(\zeta),\tau^{-1}(\zeta'))\cdot \sqrt{(\tau^{-1})'(\zeta')}\,.
\end{align}
To compute this explicitly for the strip $\Strip=\Strip_{-b,b}$, we insert \eqref{eq:tau-formulas} and get
\begin{align}
     K_\Strip(\zeta,\zeta')
     &=
     \frac{\pi}{4b}
     \frac{1}{\cosh\frac{\pi(\bar\zeta-\zeta')}{4b}}
     \,.
\end{align}
We also recall that given any orthonormal basis $\{\psi_n\}_n$ of $H^2(\Strip)$, we have $\sum_n\overline{\psi_n(\zeta)}\psi_n(\zeta')=K_\Strip(\zeta,\zeta')$. So, in particular,
\begin{align}\label{eq:K-sum}
     \sum_n|\psi_n(\zeta)|^2
     =
      \frac{\pi}{4b}
     \frac{1}{\cos\frac{\pi \,\im(\zeta)}{2b}}
     \,.
\end{align}

\begin{proposition}\label{prop2}
	Let $w\in H^\infty(\Strip)$ and $0\leq\la<1$.
     \begin{enumerate}[i)]
	  \item  If $w|_{\la\Strip}\in H^2(\la\Strip)$ (for $\la>0$) or $w_0\in L^2(\Rl)$ (for $\la=0$), then $R_{w,\la}$ is Hilbert-Schmidt, with Hilbert-Schmidt norm
	  \begin{align}
	       \|R_{w,\la}\|_2 = \sqrt{\frac{\pi}{4b\cos\frac{\pi\la}{2}}}\cdot\|w\|_{\la\Strip},\;\;\la>0,\qquad
	       \|R_{w,0}\|_2 = \sqrt{\frac{\pi}{4b}}\cdot\|w\|_2\,.
	  \end{align}
	  \item If $w$ is non-vanishing and rapidly decreasing in the sense that for any $k\in\Nl$, 
	  \begin{align}\label{eq:w-rapid}
	       \sup_{{\te\in\Rl}\atop{-\la b\leq \mu\leq\la b}}\left(|w(\te+i\mu)|\,(1+\te^2)^k\right)<\infty
	       \,,
	  \end{align}
	  then the approximation numbers of $R_{w,\la'
	  }$ satisfy for any $N\in\Nl$
	  \begin{align}
	       \sup_{n\in\Nl} \left(n^N a_n(R_{w,\la})\right) < \infty \,.
	  \end{align}
	  \item If $w(\te)\to c$, $c\neq0$, as $\te\to\infty$ or $\te\to-\infty$, then $R_{w,\la}$ is not compact.
     \end{enumerate}
\end{proposition}
\begin{proof}
     {\em i)} Let $\{\psi_n\}_n$ be some orthonormal basis of $H^2(\Strip)$, and $0<\la<1$. Then, using \eqref{eq:K-sum},
     \begin{align*}
	  \sum_n\|R_{w,\la}\psi_n\|_{\la\Strip}^2
	  &=
	  \frac{1}{2\pi}\,\sum_n\int_\Rl d\te\left( |w_{-\la b}(\te)|^2|\psi_{n,-\la b}(\te)|^2+|w_{\la b}(\te)|^2|\psi_{n,\la b}(\te)|^2\right)
	  \\
	  &=
	  \frac{1}{8b}\,\int_\Rl d\te\left( \frac{|w(\te-i\la b)|^2}{\cos\frac{\pi \la}{2}}+\frac{|w(\te+i\la b)|^2}{\cos\frac{\pi \la}{2}}\right)
	  \\
	  &=
	  \frac{\pi}{4b\cos\frac{\pi\la}{2}}\cdot\|w\|_{\la\Strip}^2\,.
     \end{align*}
     Since $\|R_{w,\la}\|_2^2={\rm Tr}(R_{w,\la}^*R_{w,\la})$, this finishes the proof for $\la>0$. The argument for $\la=0$ is analogous.
     
     {\em ii)} By \eqref{eq:anscale}, it is sufficient to show the claim for $\la>0$. Let $k\in\Nl$. We find intermediate strip regions $\Strip\supset\la_1\Strip\supset\la_2\Strip_2\supset...\supset\la_{k-1}\Strip\supset\la\Strip$ such that in each inclusion, the closure of the smaller strip is contained in the larger strip. In view of the assumption on $w$, we may furthermore write our weight as a product $w=w_1\cdot w_2\cdots w_k$, with $w_1,...,w_k\in H^2(\Strip)$ (We can take $w_j:=w^{1/k}$ for $j=1,...,k$.) Thus our operator can be written as 
     \begin{align*}
	  R_{w,\la}=R_{\la/\la_{k-1},w_k}\cdot R_{\la_{k-1}/\la_{k-2},w_{k-1}}\cdots R_{\la_1,w_1}\,.
     \end{align*}
     By part $i)$, each of the $k$ factors is Hilbert-Schmidt. That is, $R_{w,\la}$ can be written as a product of an arbitrary number of Hilbert-Schmidt operators. This implies the claim by standard estimates on approximation numbers \cite{Pietsch:1972}.

     {\em iii)} We consider the case that $w(\te)\to c\neq0$ as $\te\to+\infty$, the opposite limit is analogous. By \eqref{eq:anscale}, it is sufficient to consider the case $\la=0$.
     
     For non-zero $ f \in H^2(\Strip)$, we consider the sequence $ f _n(\zeta):= f (\zeta-n)$. To show that $R_{w,0}$ is not compact, we show that $\{R_{w,0} f _n\}_n$ has no convergent subsequence. In fact, by dominated convergence we have
     \begin{align*}
	  \|R_{w,0} f _n\|_2^2
	  =
	  \int_\Rl d\te\,|w(\te)|^2| f (\te-n)|^2
	  \to c^2\| f _0\|^2_2
	  \,.
     \end{align*}
     Since $c\neq0$ and $ f \neq0$, this limit is non-zero, i.e. $\|R_{w,0} f _n\|_2\geq c_0>0$ for sufficiently large $n$. On the other hand, we have 
     \begin{align*}
	  |\langle R_{w,0} f _{n_1},R_{w,0} f _{n_2}\rangle_2|
	  \leq
	  \|w\|_\infty^2
	  \int_\Rl d\te\,| f (\te)|\cdot| f (\te+n_1-n_2)|\,,
     \end{align*}
     and this converges to $0$ for $n_1-n_2\to\infty$ by the falloff properties of Hardy space functions. Thus for large $n_1$, $n_2$, $|n_1-n_2|$, the vectors $R_{w,0} f _{n_1}$, $R_{w,0} f _{n_1}$ have approximately identical non-zero length and are approximately orthogonal to each other. Thus $\{R_{w,0} f _n\}_n$ can have no convergent subsequence.
\end{proof}

The situations described in item {\em ii)} and {\em iii)} of this proposition fit to the weights appearing in the massive and massless irreducible Poincar\'e representations, introduced in the previous section. To see this, we need to translate the weights $w_{m,x}\in H^\infty(\Strip_{0,\pi})$ \eqref{eq:rep-weights} to the symmetric strip $\Strip_{-\pi/2,\pi/2}$ by shifting their argument $\zeta\to\zeta+\frac{i\pi}{2}$. In the massive case $m>0$, this results in the weight (denoted by the same symbol)
\begin{align}\label{eq:m>0-weight}
     w_{m,x}(\zeta)
     =
     e^{m(-x_+e^\zeta+x_-e^{-\zeta})}
     \,,\quad 
     \zeta\in\Strip_{-\pi/2,\pi/2}\,.
\end{align}
Taking into account that $\pm x_\pm>0$, it is apparent that $w$ is rapidly decreasing on any strip of the form $\la\Strip$, $0<\la<1$. Thus, by part {\em ii)} of the preceding proposition, the approximation numbers of $R_{w,\la}$ are rapidly decreasing for any $0\leq\la<1$.

In the massless case, however, the weight becomes 
\begin{align}\label{eq:m=0-weight}
     w_{0\pm,x}(\zeta)=e^{\mp x_\pm \,e^{\pm\zeta}}\,,\quad \zeta\in\Strip_{-\pi/2,\pi/2}\,,
\end{align}
which converges to 1 as $\zeta\to\mp\infty$. Thus, by item {\em iii)}, the weighted restriction operator $R_{w,0}$ is not compact, which fits with the known result that in the massless case, $\Xi_{x,\mu}$ is not compact.

\bigskip

The estimates of Prop.~\ref{prop2}~$ii)$ are rather imprecise, and can be improved for specific weights $w$. To derive sharper bounds, let us switch to the unit disc~$\Disc$ and make contact with the setting of Section~\ref{Section:general}.

By the very definition of the scalar product $\langle\cdot,\cdot\rangle_\Strip$, the mapping
\begin{align}
     V:H^2(\Strip)\to H^2(\Disc)\,,\qquad V\psi := \sqrt{\tau'}\cdot(\psi\circ\tau)\,,
\end{align}
is unitary, with inverse $V^{-1}f=\sqrt{(\tau^{-1})'}\cdot(f\circ\tau^{-1})$. For the scaled strip $\la\Strip$, $0<\la<1$, we have the Riemann map $\tau_\la:=\la\tau$, with corresponding unitary $V_\la$. With the help of these unitaries, we can transfer the operators $R_{w,\la}$ to the Hardy space $H^2(\Disc)$ on the disc.

\medskip

Let us introduce some notation first. We define $\varphi_\la=g^{-1}\circ \gamma_\la\circ g$, with $0<\la<1$ and 
$$g(z)=\frac{1+z}{1-z},\quad \gamma_{\la}(u)=u^\la \hbox{\ for}\; \re\, u>0$$ as the lens map of parameter $\la$. Explicitly,
\begin{align}\label{eq:lensmap}
	\varphi_{\la}(z)=\frac{(1+z)^\la -(1-z)^\la}{(1+z)^\la +(1-z)^\la}\,.
\end{align}

\begin{lemma}
	The operator $R_{w,\la}$, $0<\la<1$, is unitarily similar to a weighted lens map composition operator $M_{\hat w}C_{\varphi_\la}$ on $H^2(\Disc)$, with lens map symbol $\varphi_\la$ \eqref{eq:lensmap}, and weight
	\begin{align}\label{eq:www}
		\hat{w}_\la(z):=\sqrt{\la}\,m_\la(z)\, w(\la\,\tau(z))
		\,,\qquad 
		m_{\la}(z)
		:=
		\left(\frac{1-\varphi_\la(z)^2}{1-z^2}\right)^{1/2}
		\,.
	\end{align}
\end{lemma}
\begin{proof}
	The disc operator unitarily similar to $R_{w,\la}$ is $V_\la R_{w,\la}V^{-1}:H^2(\Disc)\to H^2(\Disc)$. Inserting the definitions yields
     \begin{align*}
	  (V_\la R_{w,\la}V^{-1}f)(z)
	  &=
	  \sqrt{\la\,\tau'(z)}\sqrt{(\tau^{-1})'(\la\,\tau(z))}\cdot w(\la\,\tau(z))\cdot f(\tau^{-1}(\la\,\tau(z)))
	  \\
	  &=
	  \sqrt{\la}\,\sqrt{\frac{\tau'(z)}{\tau'(\tau^{-1}(\la\,\tau(z)))}}\cdot w(\la\,\tau(z))\cdot f(\tau^{-1}(\la\,\tau(z)))
	  \,.
     \end{align*}
     We note that
     \begin{align}
		\tau^{-1}(\la\cdot\tau(z))
		&=
		\tanh(\la\cdot\arctanh(z))
		=
		\varphi_\la(z)
	\end{align}
	is the lens map. Inserting the explicit form of $\tau'$ \eqref{eq:tau-formulas} then yields the claimed result.
\end{proof}

The weight \eqref{eq:www} consists of the unbounded universal factor $m_\la$ (which diverges as $z\to\pm1)$ and the bounded factor $w\circ\la\tau$. For the factor $m_\la$, it is easy to show that 
\begin{align}\label{eq:m-bound}
	\left|m_\la(z)\right|
	\leq
	2
	\left|1-z^2\right|^{\frac{1}{2}(\la-1)}
	\,,\qquad
	z\in\Disc\,.
\end{align}

In the case of the ``massive weight'' \eqref{eq:m>0-weight}, we obtain after a short calculation
\begin{align}\label{eq:themassiveweight}
	w(\la\tau(z))
	\leq
	e^{-\big(s_+\,g(z)^{\la}+s_-\,g(-z)^{\la}\big)}
	\,,
\end{align}
with the parameters $s_\pm:=\frac{m\pi}{2}|x_\pm|>0$. Whereas this function decreases to zero quite fast as $z$ approaches $\pm1$, the two contact points of the lens map with the boundary of $\Disc$, the function $m_\la$ only diverges mildly at these points. For the purposes of estimating approximation numbers, we may therefore reduce the parameters $s_\pm$ a little to compensate the divergent factor $m_\la$, and consider weighted lens map composition operators with a weight of the form \eqref{eq:themassiveweight} instead. Such operators will be studied in the following section.

\bigskip

To conclude this section, we also transfer the operators $R_{w,0}$ \eqref{eq:Rw0}, corresponding to $\la=0$, to the disc. Whereas for $\la>0$, we obtained weighted composition operators on the disc, the case $\la=0$ corresponds to Carleson embeddings.

To see this, let us define on $\Disc$ the measure 
\begin{align}\label{eq:measure}
	\mu_w(x+iy):=|\hat w_1(x)|^2\delta(y) dx\,dy\,,\qquad w_1(z):=w(\tfrac{4b}{\pi}\arctanh z)\,,
\end{align}
supported on the real diameter of the disc.

\begin{lemma}
	Let $w\in H^\infty(\Strip_{-b,b})$.
	\begin{enumerate}
		\item The measure $\mu_w$ \eqref{eq:measure} is a Carleson measure, i.e. $H^2(\Disc)\subset L^2(\Disc,d\mu_w)$, and the embedding $J_w:H^2(\Disc)\hookrightarrow L^2(\Disc,d\mu_w)$ is bounded.
		\item The operator $R_{w,0}$ is unitarily similar to the embedding $J_w$.
	\end{enumerate}
\end{lemma}
\begin{proof}
	$i)$ Since the weight $\hat w_1$ is bounded on $\Disc$, we can estimate the measure of a Carleson box as, $\xi\in\mathbb T$, $0<h\leq1$,
	\begin{align*}
		\mu[S(\xi,h)]\leq \mu[S(1,h)]
		=
		\int_{1-h}^h|\hat w_1(x)|^2\,dx
		\leq
		\|\hat w_1\|_\infty^2\cdot h
	\end{align*}
	for $\re\,\xi\geq0$, and analogously for $\re\,\xi<0$. Thus $\mu$ is a Carleson measure (see Thm.~\ref{thm:carleson}), which implies the remaining statements in $i)$.
	
	$ii)$ We define the unitaries
	\begin{align}
		V_0:L^2(\Rl)\to L^2((-1,1),dx)
		\,,\qquad
		(V_0\psi)(x)
		:=
		\frac{\psi\left(\frac{4b}{\pi}\arctanh(x)\right)}{\sqrt{\frac{\pi}{4b}(1-x^2)}}
	\end{align}
	and
	\begin{align}
		\tilde V:L^2((-1,1),dx)\to L^2(\mu)
		\,,\qquad
		(\tilde V\psi)(x)
		:=
		\frac{\psi(x)}{\hat w_1(x)}\,.
	\end{align}
	Then inserting the definitions yields, $f\in H^2(\Disc)$,
	\begin{align*}
		(\tilde VV_0R_{w,0}V^{-1}f)(x)
		&=
		\frac{1}{\hat w_1(x)}\frac{w(\frac{4b}{\pi}\arctanh(x))\,(V^{-1}f)(\frac{4b}{\pi}\arctanh(x))}{\sqrt{\frac{\pi}{4b}(1-x^2)}}
		\\
		&=
		f(x)\,.
	\end{align*}
\end{proof}

\section{Weighted lens map composition operators}\label{Section:lensmaps}

We now turn to weighted {\em lens map} composition operators on $H^2(\Disc)$ as particular examples of weighted composition operators. As in the previous section, we define the analytic map $\varphi=\varphi_\lambda=g^{-1}\circ \gamma_\lambda\circ g$, with $0<\lambda<1$ and 
$$g(z)=\frac{1+z}{1-z},\quad \gamma_{\lambda}(u)=u^\lambda \hbox{\ for}\; \re\, u>0$$ as the lens map of parameter $\lambda$. Explicitly,
$$\varphi_{}(z)=\frac{(1+z)^\lambda -(1-z)^\lambda}{(1+z)^\lambda +(1-z)^\lambda}\cdot$$

\noindent Recall that the image $\varphi(\D)\subset \D$ has exactly two non-tangential contact points with the unit circle at $1$ and $-1$, and in particular verifies $\Vert \varphi\Vert_\infty=1$. Also  observe that $\varphi(-z)=-\varphi(z)$ for all $z\in \D$.\\

We choose a weight of the form motivated by the applications outlined in the previous section, namely $w=w_0\circ \varphi$ where \footnote{In comparison to \eqref{eq:themassiveweight}, we have set the inessential parameters $s_\pm$ to 1.}
\begin{equation}\label{choix}w_{0}(z)=\exp\Big[-\Big(\frac{1+z}{1-z}\Big)^\lambda\Big]\exp\Big[-\Big(\frac{1-z}{1+z}\Big)^\lambda\Big]=:w_{1}(z)w_{-1}(z)\,\end{equation}

We have $\max(\Vert w_1\Vert_\infty, \Vert w_{-1}\Vert_\infty)\leq 1$, and note that  $w_0$ tends to $0$ quite rapidly as $z\to\pm1$. Namely, if $z\in \D$:
\begin{equation}\label{majw1} \re\, z\geq 0 \Rightarrow|w_{0}(z)|\leq  |w_{1}(z)|\leq \exp\Big(-\frac{\delta}{|1-z|^\lambda}\Big)\end{equation}
\begin{equation}\label{majw2} \re\, z\leq 0 \Rightarrow|w_{0}(z)|\leq  |w_{-1}(z)|\leq \exp\Big(-\frac{\delta}{|1+z|^\lambda}\Big)\end{equation}
where $\delta=\cos (\lambda\pi/2)$. Indeed, if for example $\re\, z \geq 0$, we see that 
$$\re \Big(\frac{1+z}{1-z}\Big)^\lambda\geq \delta\Big|\frac{1+z}{1-z}\Big|^\lambda\hbox{\ with}\ |1+z|\geq \re(1+z)\geq 1.$$
We take for $w_0$ this double product to take both contact points $\pm 1$ of $\varphi(\D)$ with the unit circle into account. We note in passing (this is general) that 
$$C_\varphi M_{w_0}=M_{w_{0}\circ \varphi}\,C_\varphi=M_w C_\varphi.$$

We can now state one of our main  theorems. The positive constants $0<c<C$ are allowed to change from one line to another in what follows.

\begin{theorem}\label{onemain} Let   $T=M_w C_\varphi:H^2\to H^2$ where $w$ and $\varphi$ are as above. Then $T$ is compact, and more precisely:\begin{enumerate}
\item One nearly has exponential decay, namely
 $a_{n}(T)\leq Ce^{-c\frac{n}{\log n}}$.
\item The previous estimate is optimal: $a_{n}(T)\geq c\,e^{-C\frac{n}{\log n}}$.
\end{enumerate}
\end{theorem}
\newpage
\begin{proof} We begin with the upper bound, following the strategy of \cite{LefevreLiQueffelecRodriguez-Piazza:2012}. If $\gamma(t)=\varphi_{}(e^{it})$, one easily checks \cite[Lemma~2.5]{LefevreLiQueffelecRodriguez-Piazza:2012} that 
\begin{equation}\label{chec1}|t|\leq \frac{\pi}{2}\Rightarrow 1-|\gamma(t)|\approx|1-\gamma(t)|\approx |t|^\lambda\end{equation}
\begin{equation}\label{chec2}\frac{\pi}{2}\leq |t|\leq \pi \Rightarrow 1-|\gamma(t)|\approx|1+\gamma(t)|\approx (\pi-|t|)^\lambda.\end{equation} 
We fix an integer $N\geq 2$. Let $B$ be the Blaschke product 
$$B(z)=\prod_{1\leq k\leq \log N} \Big(\frac{z-p_k}{1-\overline{p_{k}}z}\,\frac{z-\overline{p_k}}{1-p_{k}z}\Big)^N\,\prod_{1\leq k\leq \log N}\Big(\frac{z+p_k}{1+\overline{p_{k}}z}\,\frac{z+\overline{p_k}}{1+p_{k}z}\Big)^N$$$$=:B_{1}(z)B_{-1}(z)$$
where $p_k=\gamma(t_k)$ and $t_k=\frac{\pi}{ 2}2^{-(k-1)/\lambda}$. This is a Blaschke product of length $\leq 4N\log N$. Observe that $\overline{B_{\pm1}(z)}=B_{\pm1}(\overline{z})$. Next, set $E=BH^2$,  a subspace of codimension $<[4N\log N+1]=:M$ (with $[.]$ the integer part). If $f=Bg\in E$, with $\Vert f\Vert=\Vert g\Vert=1$, we have (remembering (\ref{alt}) and since $w=w_0\circ\varphi$)
$$\Vert T(f)\Vert^2=\int_{\T}|w_{0}(\varphi(u))|^2|B(\varphi(u))|^2 |g(\varphi(u))|^2 dm(u)= \int_{\D}|w_0|^2 |B|^2|g|^2 dm_{\varphi}.$$
 Now,  using again Carleson's embedding theorem for the measure 
 $$\mu=|w_0|^2|B|^2dm_{\varphi},$$ as in the proof of Theorem \ref{geup}, and the non-tangential behavior of $\varphi$ near the points $\pm 1$ (allowing us to ignore the Carleson windows centered elsewhere than in $\pm 1$ as in \cite[p.~809]{LefevreLiQueffelecRodriguez-Piazza:2012}), we get with help of (\ref{gelf}):
 $$a_{M}(T)^2\lesssim \sup_{0<h<1}\Big(\frac{I_{1}(h)}{h}+\frac{I_{-1}(h)}{h}\Big)$$ where
\begin{align*}
	I_{1}(h)
	&=
	\int_{|\gamma(t)-1|\leq Ch} |B_{1}(\gamma(t))|^2|w_{1}(\gamma(t))|^2dt\\
	I_{-1}(h)
	&=\int_{|\gamma(t)+1|\leq Ch} |B_{-1}(\gamma(t))|^2|w_{-1}(\gamma(t))|^2dt. 
\end{align*}
Actually, the term $I_{1}(h)$ takes care of the Carleson boxes $S(\xi,h)$ for which $\re\, \xi\geq 0$ and the term $I_{1}(h)$ of those for which $\re\, \xi< 0$. We will estimate only $I_{1}(h)$, the estimate being similar for $I_{-1}(h)$.\\
  By interpolation, we can assume $h=2^{-q/\lambda}$ with $q$ a non-negative integer, and we separate two cases:
 
\medskip

\noindent {\bf Case 1: $q>\log N$}. Then,  majorizing $|B_{1}(\gamma(t))|$ by $1$  and using (\ref{majw1}) as well as (\ref{chec1}), we get  
$$h^{-1}I_{1}(h)\lesssim h^{-1}\int_{|t|^\lambda\leq Ch}|w_{1}(\gamma(t))|^2dt\lesssim h^{-1}\int_{|t|^\lambda\leq Ch}\exp\Big(-\frac{C}{|t|^{\lambda^2}}\Big)dt$$$$\lesssim h^{\frac{1}{\lambda}-1}\exp\Big(-\frac{C}{h^\lambda}\Big)\lesssim \exp(-C2^{q})\lesssim \exp(-CN).$$
{\bf Case 2: $q\leq \log N$}. Then, we majorize $|w_{1}(\gamma(t))|$ by $1$ and estimate $|B_{1}(\gamma(t))|$ more accurately, with help of an obvious modification of Lemma~2.6 in \cite{LefevreLiQueffelecRodriguez-Piazza:2012}, which we recall:
\begin{lemma}\label{recal} Set\  $t_k=\frac{\pi}{ 2}2^{-(k-1)/\lambda}$. Then 
$$t_{[\log N]}\leq t\leq t_1\Rightarrow |B_{1}(\gamma(t))|\leq \chi^N$$ where $[.]$ denotes the integer part and where  $\chi<1$ only depends on $\lambda$.
\end{lemma} 
We can now finish the estimate (note that $t_1=\pi/2$):
$$h^{-1}I_{1}(h)\lesssim h^{-1}\int_{0}^{t_{[\log N]}} |w_{1}(\gamma(t))|^2dt+h^{-1} \int_{t_{[\log N]}}^{t_1}|B_{1}(\gamma(t))|^2dt.$$
The first term is estimated as in Case 1. And since $q\leq \log N$, the second sum is dominated by 
$2^{q/\lambda} \chi^{2N}\lesssim N^{1/\lambda}\exp(-CN)\lesssim \exp(-C'N)$. 
Putting both cases together, we get 
$\sup_{0<h<1}h^{-1}I_{1}(h)\lesssim \exp(-CN)$ and finally, with $M=[4N\log N+1]$ and using once more (\ref{gelf}): 
$$a_{M}^{2}\lesssim e^{-CN}.$$
 Inverting, interpolating, and using that $(a_p)_{p\geq 1}$ is non-increasing, we finally get for all positive integers $n$:
$$a_n\lesssim e^{-C\frac{n}{\log n}}$$as claimed. This ends the proof of the upper bound.\\

  To establish the lower bound, we first observe that 
$$w_{0}(z)=\exp(-\gamma_{\lambda}g(z))\exp(-\gamma_{\lambda}g(-z)).$$
 Since $\gamma_\lambda \gamma_\lambda=\gamma_{\lambda^2}$, we get  $\gamma_{\lambda}g\varphi=\gamma_{\lambda^2}\,g$, and an explicit formula for $w$ is 
\begin{equation}\label{expl} w(z)=\exp\Big[-\Big(\frac{1+z}{1-z}\Big)^{\lambda^2}\Big] \exp\Big[-\Big(\frac{1-z}{1+z}\Big)^{\lambda^2}\Big]=:w'_{1}(z)w'_{-1}(z). \end{equation}
We now apply Lemma \ref{berber}. To that effect, we must  make a good choice of the $u_j$'s. As in \cite{LiQueffelecRodriguez-Piazza:2012} for lens maps, we choose $u_j=1-e^{-j\eps}$ where $\eps>0$ has to be adjusted, and $v_j=\varphi(u_j)$. We know from \cite[Lemma~6.5]{LiQueffelecRodriguez-Piazza:2012_3} that $I_v\leq \exp(C/\eps)$, and we have $\sqrt{1-u_{j}^{2}}\geq c\,e^{-n\eps}$. Moreover, since
$$\Big(\frac{1+u_j}{1-u_j}\Big)^{\lambda^2}\leq  \Big(\frac{2}{1-u_j}\Big)^{\lambda^2},$$ we see that
$$\inf_{1\leq j\leq n} |w'_{1}(u_j)|\geq \exp(-Ce^{n\varepsilon}).$$

 \noindent And clearly $\inf_{1\leq j\leq n} |w'_{-1}(u_j)|\geq e^{-1}$. So that $\inf_{1\leq j\leq n} |w(u_j)|\geq \exp(-Ce^{n\eps})$ (recall that $w=w'_{1}w'_{-1}$). Lemma \ref{berber} now gives us  
$$a_{n}(T)\gtrsim \exp\Big[-C\big(e^{n\eps}+n\eps +\frac{1}{\eps}\big)\Big]\gtrsim \exp\Big[-C\big(e^{n\eps} +\frac{1}{\eps}\big)\Big].$$
We finally adjust $\eps=\frac{1}{2}\frac{\log n}{n}$ to get 
$$a_{n}(T)\gtrsim \exp\Big[-C\big(\sqrt{n} +\frac{n}{\log n}\big)\Big]\gtrsim \exp\Big[-C\frac{n}{\log n}\Big].$$
This ends the proof of Theorem  \ref{onemain}. 
\end{proof}

\noindent{\bf Acknowledgements:} G.~Lechner would like to thank O.~Bandtlow for discussions and pointing out the article \cite{LefevreLiQueffelecRodriguez-Piazza:2012}, which initiated this collaboration.\\
 L.~Rodr\'iguez-Piazza was also supported by the research project MTM2015-63699-P (Spanish MINECO and FEDER funds).

\end{document}